\documentclass[amscd,amssymb,11pt]{amsart}

\usepackage{amssymb}
\usepackage{graphicx}
\usepackage{xcolor}
\usepackage[all]{xy}

\newtheorem{thm}{Theorem}[section]
\newtheorem{lem}[thm]{Lemma}
\newtheorem{cor}[thm]{Corollary}
\newtheorem{prop}[thm]{Proposition}
\newtheorem{conj}[thm]{Conjecture}

\theoremstyle{definition}

\newtheorem{defn}[thm]{Definition}

\newtheorem{ex}[thm]{Example}

\newtheorem{addendum}[thm]{Addendum}

\theoremstyle{remark}

\newtheorem{rem}[thm]{Remark}
\newtheorem{rems}[thm]{Remarks}

\numberwithin{equation}{section}

\newcommand{\thmref}[1]{Theorem~\ref{#1}}

\newcommand{\secref}[1]{\S\ref{#1}}

\newcommand{\propref}[1]{Proposition~\ref{#1}}
\newcommand{\lemref}[1]{Lemma~\ref{#1}}

\newcommand{\exref}[1]{Example~\ref{#1}}

\newcommand{\addendref}[1]{Addendum~\ref{#1}}

\newcommand{\hocolim}{\operatorname*{hocolim}}

\newcommand{\colim}{\operatorname*{colim}}

\newcommand{\Hom}{\operatorname{Hom}}
\newcommand{\Ext}{\operatorname{Ext}}

\newcommand{\Map}{\operatorname{Map}}

\newcommand{\im}{\operatorname{im}}
\newcommand{\Fun}{\operatorname{Fun}}

\newcommand{\Mod}{\text{-Mod}}
\newcommand{\Alg}{\text{-Alg}}

\newcommand{\uspace}{\underline{\text{\hspace{.13in}}}}

\newcommand{\A}{{\mathcal  A}}
\newcommand{\B}{{\mathcal  B}}
\newcommand{\C}{{\mathcal  C}}
\newcommand{\D}{{\mathcal  D}}
\newcommand{\Com}{{\mathcal  Com}}

\newcommand{\Oo}{{\mathcal  O}}

\newcommand{\Z}{{\mathbb  Z}}
\newcommand{\R}{{\mathbb  R}}
\newcommand{\Cx}{{\mathbb  C}}
\newcommand{\Ham}{{\mathbb  H}}
\newcommand{\Oct}{{\mathbb  O}}

\newcommand{\Sinfty}{\Sigma^{\infty}}
\newcommand{\Oinfty}{\Omega^{\infty}}

\newcommand{\sm}{\wedge}

\newcommand{\ra}{\rightarrow}
\newcommand{\xra}{\xrightarrow}
\newcommand{\la}{\leftarrow}
\newcommand{\xla}{\xleftarrow}

\newcommand{\lra}{\longrightarrow}

\begin{document}

\title[Hurewicz Homomorphisms]{Adams filtration and generalized Hurewicz maps for infinite loopspaces}

\author[Kuhn]{Nicholas J.~Kuhn}

\address{Department of Mathematics \\ University of Virginia \\ Charlottesville, VA 22904}

\email{njk4x@virginia.edu}


\date{July 2, 2018.}


\begin{abstract}  We study the Hurewicz map
$$ h_*: \pi_*(X) \ra R_*(\Oinfty X)$$
where $\Oinfty X$ is the $0$th space of a spectrum $X$, and $R_*$ is the generalized homology theory associated to a connective commutative $S$--algebra $R$.

We prove that the decreasing filtration of the domain associated to an $R$--based Adams resolution is compatible with a filtration of the range associated to the augmentation ideal filtration of the augmented commutative $S$--algebra $\Sinfty (\Oinfty X)_+$.  The proof of our main theorem makes much use of composition properties of this filtration and its interaction with Topological Andr\'e--Quillen homology.

An application is a Connectivity Theorem: Localize away from $(p-1)!$ and suppose $X$ is $(c-1)$--connected with $c>0$.  If $\alpha \in \pi_*(X)$ has Adams filtration $s$ and $|\alpha| < cp^s$, then $h_*(\alpha)=0 \in R_*(\Oinfty X)$.   When specialized to mod $p$ homology, this implies a Finiteness Theorem: if $H^*(X;\Z/p)$ is finitely presented as a module over the Steenrod algebra, then the image of the Hurewicz map in $H_*(\Oinfty X;\Z/p)$ is finite.  We illustrate these theorems with calculations of the mod 2 Hurewicz image of $BO$, its connected covers, and $\Oinfty tmf$, and the mod $p$ Hurewicz image of all the spaces in the $BP$ and $BP\langle n \rangle$ spectra.  En route, we get new proofs of theorems of Milnor and Wilson.

In the special case when $X$ is the suspension spectrum of a space $Z$ and $R = H\Z/2$, we recover results announced by Lannes and Zarati in the 1980s, relating the Adams filtration of $\pi_*^S(Z)$ to Dyer-Lashof length in $H_*(QZ;\Z/2)$, and generalize them to all primes $p$.  For any $X$, we also get parallel results for the Hurewicz map for Morava $E$--theory, where $\pi_*(X)$ is now given the Adams-Novikov filtration.

\end{abstract}

\maketitle

\section{Introduction} \label{introduction}

In this paper we study generalized Hurewicz maps for infinite loopspaces.

This is an old and interesting problem as we illustrate with an example dating back to the 1950's.  In \cite{milnor}, J.Milnor showed that there exists a vector bundle $\xi \ra S^n$ with nonzero top Stiefel--Whitney class if and only if $n=1,2,4,8$.  A stable vector bundle $\xi \ra S^n$ corresponds to $f_\xi \in \pi_n(BO)$, and it is easily checked that $w_n(\xi)\neq 0$ exactly when $h_*(f_{\xi}) \neq 0$, where
$$h_*: \pi_*(BO) \ra H_*(BO;\Z/2)$$
is the Hurewicz map.

Milnor proves his theorem by combining work of Bott on the integrality of Pontryagin classes with work of Wu, and Bott's work took advantage of his then recently proved  periodicity theorems.  Thus Milnor is implicitly using the infinite loop structure of the space $BO$. More precisely, $BO$ identifies as the 0th space $\Oinfty bo$ of the spectrum $bo$, where $bo$ is the 0--connected cover of the real $K$--theory spectrum $KO$.  The Hurewicz map calculated by his theorem can then be written
$$ h_*: \pi_*(bo) \ra H_*(\Oinfty bo;\Z/2).$$

My goal in this paper is to introduce a new approach to Hurewicz maps of the form
$$ h_*: \pi_*(X) \ra R_*(\Oinfty X)$$
where $X$ is a connective $S$--module (i.e., a connective spectrum), and $R$ is a suitable connective commutative $S$--algebra (e.g., $H\Z/p$, $MU$, or $ko$).  Our discovery is that the $R$--based Adams filtration of the domain of $h_*$ is compatible with a decreasing filtration of the range defined and understood using modern work on topological Andr\'e--Quillen homology and Goodwillie calculus.

We note that the Hurewicz maps considered here detect as least as much, and potentially more, than their stable analogues, as there is a canonical commutative diagram
\begin{equation*}
\xymatrix{
&& R_*(\Oinfty X) \ar[d]^{\epsilon_*}  \\
\pi_*(X) \ar[rr]^-{h_*^{st}} \ar[urr]^-{h_*}  && R_*(X), }
\end{equation*}
where $\epsilon: \Sinfty \Oinfty X \ra X$ is the counit of the adjunction $(\Sinfty, \Oinfty)$.  Our work in this paper shows that one can profitably study the unstable Hurewicz map via stable information.

\begin{rem}  The study of the Hurewicz map is interesting only in positive dimensions, as $ h_*: \pi_0(X) \ra R_0(\Oinfty X)$ identifies with the canonical inclusion of the group $\pi_0(X)$ into the group ring $R_0[\pi_0(X)]$. Note also that, if $Y$ is the 0--connected cover of $X$, then the Hurewicz map for $\Oinfty X$ in positive degrees agrees with the Hurewicz map for $\Oinfty Y$.
\end{rem}

\subsection{The Lifting Theorem} \label{main thm stated sec}

Our commutative $S$--algebras will be assumed to satisfy the following hypothesis: the cofiber $R/S$ of the unit map $\eta: S\ra R$ is 0--connected.  Equivalently, $R$ is connective, and $\pi_0(S) \ra \pi_0(R)$ is onto.

The Hurewicz map $h_*$ is induced by the map of spaces
$$ h: \Oinfty X \ra \Oinfty (R \sm \Sinfty X)$$
adjoint to the map of $S$--modules
$$ \tilde h: \Sinfty \Oinfty X \ra R \sm \Sinfty \Oinfty X,$$
obtained by smashing $\eta: S \ra R$ with $\Sinfty \Oinfty X$

We define filtrations of the domain and range of $h$.

We filter the domain via the canonical functorial $R$--based Adams filtration.  Let $S(1)$ be the homotopy fiber of the unit $S \ra R$.  Given an $S$--module $X$, let $X(s) = S(1)^{\sm s} \sm X$, so that there are homotopy fibration sequences $X(s) \ra X(s-1) \ra R \sm X(s-1)$.  The decreasing filtration
$$ X=X(0) \la X(1) \la X(2) \la \dots$$
defines the Adam Spectral Sequence upon applying homotopy, and the stable Hurewicz map $h^{st}_*$ appears as an edge homomorphism. The filtration of $X$ induces the decreasing filtration
$$\pi_*(X) = F_0\pi_*(X) \supset F_1\pi_*(X) \supset F_2 \pi_*(X) \supset \dots,$$
where $F_s\pi_*(X) = \im \{\pi_*(X(s) \ra \pi_*(X)\}$, the subgroup of elements of Adams filtration at least $s$.  Since $X(s) \simeq (\Sigma^{-1}R/S)^{\sm n} \sm X$, our assumption that $R/S$ is 0--connected ensures that if $X$ is connective, so is each $X(s)$.

The filtration of the range takes more explanation.

The spectrum $\Sinfty (\Oinfty X)_+$ is naturally an augmented commutative $S$--algebra, where $Z_+$ denotes the disjoint union of a space $Z$ and a disjoint basepoint.  We define $I(X)$ to be the homotopy fiber of the augmentation $\Sinfty (\Oinfty X)_+ \ra S$.  $I(X)$ is naturally a nonunital commutative $S$--algebra.  Furthermore,
the natural composite $I(X) \ra \Sinfty (\Oinfty X)_+ \ra \Sinfty \Oinfty X$ is a weak equivalence.

Let $S\Alg$ denote the category of nonunital commutative $S$--algebras.  Then $I(X)$ admits a decreasing `augmentation ideal' filtration in $S\Alg$
$$ I(X) \la I^2(X) \la I^3(X) \la \dots $$
satisfying lots of nice properties: see \secref{aug ideal properties sec}.  In particular, if $X$ is connective, there are fibration sequences
$$ I^{d+1}(X) \ra I^d(X) \ra D_dX,$$
where $D_d(X) = X^{\sm d}_{h\Sigma_d}$, the $d$th extended power of $X$.  We will make good use of the fact that the connectivity of $I^d(X)$ agrees with that of $D_d(X)$: if $X$ is $(c-1)$-connected then $I^d(X)$ is $(cd-1)$ connected \cite[Prop.2.22]{kuhn-pereira}.

Let $\tilde h_0: I(X) \ra R \sm I(X)$ be $\eta: S\ra R$ smashed with $I(X)$.  This is a map in $S\Alg$, and corresponds to $\tilde h: \Sinfty \Oinfty X \ra R \sm \Sinfty \Oinfty X$.

Our main theorem now says that our various filtrations are `exponentially' compatible. To state this, we need a little bit more notation.  Let $S\Mod_{\geq 0}$ be the category of connective $S$--modules and let $\Fun(S\Mod_{\geq 0}, S\Alg)$ be the category of finitary functors $F:S\Mod_{\geq 0} \ra S\Alg$.  As will be detailed in \secref{ho cat section},  $\Fun(S\Mod_{\geq 0}, S\Alg)$ has a model category structure with weak equivalences being functors $F$ such that each $F(X)$ is a weak equivalence.  We let $ho\Fun(S\Mod_{\geq 0}, S\Alg)$ denote the associated homotopy category.

\begin{thm}[Lifting Theorem] \label{lifting thm}  Let $R$ be a connective commutative $S$-algebra, such that $\pi_0(S) \ra \pi_0(R)$ is onto. Localized away from $(p-1)!$,  there are natural maps $ \tilde h_s: I(X(s)) \ra R \sm I^{p^s}(X)$ in $ho\Fun(S\Mod_{\geq 0}, S\Alg)$
fitting into a commutative diagram in $ho\Fun(S\Mod_{\geq 0}, S\Alg)$:
\begin{equation*}
\xymatrix{
I(X(0)) \ar[d]^{\tilde h_0}  & I(X(1)) \ar^{\tilde h_1}[d] \ar[l] & I(X(2)) \ar^{\tilde h_2}[d] \ar[l] & \ar[l] \dots \\
R \sm I(X) & R \sm I^p(X) \ar[l] & R \sm I^{p^2}(X) \ar[l] & \ar[l] \dots .}
\end{equation*}
\end{thm}

Recalling that $I(X) \simeq \Sinfty \Oinfty X$. Adjointing, this gives a filtration of the Hurewicz map $h$.

\begin{thm}[Lifting Theorem, adjointed] \label{lifting thm adjointed} Let $R$ be a connective commutative $S$-algebra, such that $\pi_0(S) \ra \pi_0(R)$ is onto.  Localized away from $(p-1)!$, for all connective $X$, one has a natural commutative diagram in the homotopy category of based spaces
\begin{equation*}
\xymatrix{
\Oinfty X(0) \ar[d]^h  & \Oinfty X(1) \ar^{h_1}[d] \ar[l] & \Oinfty X(2) \ar^{h_2}[d] \ar[l] & \ar[l] \dots \\
\Oinfty (R \sm \Sinfty \Oinfty X) & \Oinfty (R \sm I^p(X)) \ar[l] & \Oinfty (R \sm I^{p^2}(X)) \ar[l] & \ar[l] \dots,}
\end{equation*}
and thus a natural commutative diagram
\begin{equation*}
\xymatrix{
\pi_*(X) \ar[d]^{h_*}  & \pi_*(X(1)) \ar^{h_{1*}}[d] \ar[l] & \pi_*(X(2)) \ar^{h_{2*}}[d] \ar[l] & \ar[l] \dots \\
R_*(\Oinfty X) & R_*(I^p(X)) \ar[l] & R_*(I^{p^2}(X)) \ar[l] & \ar[l] \dots .}
\end{equation*}

\end{thm}

\thmref{lifting thm} is proved in \secref{lifting thm proof sec}, using needed background material which is collected and developed in \secref{background sec}.  When $p=2$, the case in which localizing away from $(p-1)!$ has no content, the proof follows from the interesting interaction between the augmentation ideal filtration of a nonunital algebra $I$ and its Topological Andr\'e--Quillen homology $T(I)$ as applied to the nonunital commutative $S$--algebra $I(X)$.  In rough outline, the $s=1$ case follows from the properties of $T(I)$ together with the calculation \cite{basterra-mandell, kuhn-TAQtowers} that $T(I(X)) \simeq X$: see \lemref{lifting lemma 1}. The general case is then proved by induction on $s$ using the composition product on the augmentation ideal filtration constructed in \cite{kuhn-pereira}.   When $p>2$, an extra lifting lemma, \lemref{lifting lemma 2}, is needed to take care of the case $s=1$.  We prove this with a slightly delicate application of Goodwillie calculus to the functor sending an $S$--module to its connective cover.

As in \cite{kuhn-pereira}, we work in the world of the symmetric spectra of \cite{hss}.  Our categories of modules and algebras are based simplicial model categories, and in particular are tensored over based simplicial sets.  If a functor $F: \A \ra \B$ between two such categories is simplicial then it admits natural transformations $\epsilon_K(X): K \otimes F(X) \ra F(K \otimes X)$.  We have the following addendum to \thmref{lifting thm}.

\begin{addendum} \label{addendum prop}  The functors $I^d(X)$ are simplicial for all $d$.  For all based simplicial sets $K$, the diagrams
\begin{equation*}
\SelectTips{cm}{}
\xymatrix{
K \otimes I(X(s)) \ar[d]^{\epsilon_K(X(s))} \ar[rr]^{K \otimes \tilde h_s(X)} && K \otimes (R \sm I^{p^s}(X)) \ar[d]^{\epsilon_K(X)}  \\
I(K \sm X(s)) \ar[rr]^{\tilde h_s(K \sm X)} && R \sm I^{p^s}(K \sm X) }
\end{equation*}
and
\begin{equation*}
\SelectTips{cm}{}
\xymatrix{
K \otimes I^{d+1}(X) \ar[d]^{\epsilon_K(X)} \ar[r] & K \otimes I^{d}(X) \ar[d]^{\epsilon_K(X)} \ar[r] & K \sm D_d(X) \ar[d]^{\Delta}  \\
I^{d+1}(K \sm X)  \ar[r] & I^{d}(K \sm X) \ar[r] & D_d(K \sm X) }
\end{equation*}
commute in $ho\Fun(S\Mod_{\geq 0}, S\Alg)$.  Here the map labeled `$\Delta$' is induced by the diagonal $\Delta: K \ra K^{\sm d}$.
\end{addendum}

We discuss this, and add more detail and explanation, in \secref{background sec} and \secref{lifting thm proof sec}; in particular, see \secref{simplicial functor sec}.

\begin{rem}  The maps $\tilde h_s$ of the Lifting Theorem are maps of commutative $S$--algebras, and not just $S$--modules.  The implication of this for the maps $h_s$ appearing in \thmref{lifting thm adjointed} is that these maps are exponential in an appropriate sense.  In this paper we do not pursue the calculational consequences of this.
\end{rem}

\subsection{Applications: the Connectivity, Finiteness, and Atomicity Theorems}

A tidy consequence of \thmref{lifting thm adjointed} goes as follows.

\begin{thm}[Connectivity Theorem] \label{conn thm} Let $R$ be a connective commutative $S$-algebra, such that $\pi_0(S) \ra \pi_0(R)$ is onto. Localize away from $(p-1)!$.   Suppose $X$ is $(c-1)$--connected with $c>0$.  If $\alpha \in F_s\pi_*(X)$ has $|\alpha| < cp^s$, then $h_*(\alpha)=0 \in R_*(\Oinfty X)$.
\end{thm}

\begin{proof} By definition, an element $\alpha \in F_s\pi_*(X)$ is in the image of $\pi_*(X(s)) \ra \pi_*(X)$.  By \thmref{lifting thm adjointed}, the Hurewicz image of $\alpha$ will then be in the image of $R_*(I^{p^s}(X)) \ra R_*(\Oinfty X)$.  But  if $X$ is $(c-1)$--connected, then $I^{p^s}(X)$ is $cp^s-1$ connected, so that $R_*(I^{p^s}(X))$ is  zero in degrees less than $cp^s$.
\end{proof}

When $p=2$ and $s=1$, this result is a consequence of the Freudenthal suspension theorem which implies that the fiber of $\epsilon: \Sinfty \Oinfty X \ra X$ will be $(2c-1)$--connected if $X$ is $(c-1)$--connected.  But all other cases seem to be results of a sort that haven't  before appeared in the literature.

Specializing to the classical case when $R=H\Z/p$, the Connectivity Theorem has some striking consequences, as we now describe.  Section \ref{mod p applications sec} will have details.

A very general result goes as follows.

\begin{thm}[Finiteness Theorem] \label{finiteness thm} Let $X$ be a bounded below spectrum of finite type\footnote{So $X$ might have nontrivial homotopy groups in a finite number of negative degrees.}. If $H^*(X;\Z/p)$ is a finitely presented module over the mod $p$ Steenrod algebra $A$, then $h_*: \pi_{*}(X) \ra H_*(\Oinfty X;\Z/p)$ has finite dimensional image in positive degrees.
\end{thm}
We sketch the idea here; see \secref{finiteness section} for the full proof.   One first checks that one can reduce to the case when $X$ is 0--connected.  In this case, the Connectivity Theorem says that the kernel of $h_*$ includes all elements represented in the $E_{\infty}$--term of the Adams spectral sequence lying over a  logarithmic curve, using the standard way of depicting the Adams spectral sequence. Meanwhile, under the hypothesis of finite presentation, the Adams $E_2$--term vanishes under a line of positive slope, and thus the same holds at $E_{\infty}$.  Only a finite number of terms can be both over this line and under the logarithmic curve.

Essentially the same argument proves the next theorem.  To explain this, we need some notation.  If $X$ is a spectrum, we let $X_c$ denote its $c^{\text{th}}$ space: the space $\Oinfty \Sigma^c X$.  We observe that the mod p Hurewicz maps of the various spaces $X_c$ and the spectrum $X$ are related via the natural maps $\sigma_c: \Sigma X_c \ra X_{c+1}$ and $\epsilon_c: \Sigma^{-c}\Sinfty X_c \ra X$.  More precisely, the maps $\sigma_c$ induce epimorphisms from
$$ \im\{\pi_*(X) = \pi_{*+c}(X_c) \xra{h_*} H_{*+c}(X_c;\Z/p)\}$$
to
$$\im\{\pi_*(X) = \pi_{*+c+1}(X_{c+1}) \xra{h_*} H_{*+c+1}(X_{c+1};\Z/p)\},$$
and the maps $\epsilon_c$ induce an isomorphism from the colimit of these epimorphisms to $\im\{\pi_*(X) \xra{h^{st}_*} H_{*}(X;\Z/p)\}$.

\begin{thm}[Atomicity Theorem] \label{atomicity thm} Let $X$ be a connective spectrum such that $H^*(X;\Z/p)$ is a finitely presented $A$-module.  Then for all large enough $c$, $\epsilon_c$ induces an isomorphism from $ \im\{\pi_{*+c}(X_c) \xra{h_*} H_{*+c}(X_c;\Z/p)\}$ to $\im\{\pi_*(X) \xra{h^{st}_*} H_{*}(X;\Z/p)\}$.  If $H^*(X;\Z/p)$ is also a cyclic $A$-module, then, for such $c$, $\im\{\pi_{*+c}(X_c) \xra{h_*} \widetilde H_{*+c}(X_c;\Z/p)\}$ is just a one dimensional subspace, and thus $X_c$ is atomic: it can't be written as a product of two spaces in a nontrivial way.
\end{thm}

This is proved in \secref{atomic section}.

\begin{rems} For a discussion of aspects of atomicity see \cite{adams kuhn}.  Andy Baker notes that the last part of this theorem says that when $H^*(X;\Z/p)$ is a cyclic $A$-module, then, for large $c$, $X_c$ is {\em minimally atomic} in the sense studied by Baker and Peter May in \cite{baker may}.  We note with curiosity that the spectra $X$ for which these last two theorems apply are precisely the {\em fp-spectra} studied, using Brown-Comenetz duality and chromatic localization, by Mark Mahowald and Charles Rezk in \cite{mahowald rezk}.
\end{rems}

Here are some specific examples, illustrating these last two theorems.

The mod 2 Adams spectral sequence charts for the connective real $K$--theory spectrum $ko$ are well known and easy to compute, starting from the calculation that $H^*(ko;\Z/2) = A//A(1)$, where $A(1)$ is the subalgebra generated by $Sq^1$ and $Sq^2$.  One can then easily deduce similar information for the connective covers of $ko$.  The Connectivity Theorem plus inspection of the Adams $E_2$--pages immediately implies Milnor's theorem and higher connected analogues.

\begin{thm} \label{bo thm} The image of $h_*: \pi_*(BO) \ra H_*(BO;\Z/2)$ is one dimensional in degrees 1,2,4, and 8. The image of $h_*: \pi_*(BSO) \ra H_*(BSO;\Z/2)$ is one dimensional in degrees 2,4, and 8. The image of $h_*: \pi_*(BSpin) \ra H_*(BSpin;\Z/2)$ is one dimensional in degrees 4 and 8. For $c\geq 8$ with $c \equiv 0,1,2$, or $4 \mod 8$, the mod 2 Hurewicz image for the $(c-1)$--connected cover\footnote{This includes $BString$ when $c=8$.} of $BO$ is just the bottom one dimensional subspace in degree $c$.
\end{thm}

Details are in  \secref{bo section}.

An analysis of the Adams spectral sequence for the connective topological modular forms spectrum $tmf$ has also been done \cite{henriques}, starting from the calculation that $H^*(tmf;\Z/2) = A//A(2)$, where $A(2)$ is the subalgebra generated by $Sq^1$, $Sq^2$, and $Sq^4$.  We use this work together with the Connectivity Theorem to quite easily show the following.

\begin{thm} \label{tmf thm} In positive degrees, $h_*: \pi_*(tmf) \ra H_*(\Oinfty tmf;\Z/2)$ has five dimensional image, with basis the image of elements $\eta, \eta^2, \nu, \nu^2, c_4$ of degrees 1,2,3,6,8.
\end{thm}

This is proved in \secref{tmf section}.

Fixing a prime $p$, the $n$th Johnson--Wilson spectra $BP\langle n \rangle$ is a $p$--local spectrum satisfying $H^*(BP\langle n \rangle;\Z/p) = A//E(n)$, where $E(n)$ is the subalgebra generated by the Milnor primitives $Q_0, \dots, Q_n$ \cite[Prop.1.7]{wilson II}.  Just using this information, the Connectivity Theorem easily implies our next result.

\begin{thm} \label{BPn thm} If $c> 2(p^n-1)/(p-1)$, then the mod $p$ Hurewicz image of $BP\langle n \rangle_c$ is just the bottom $\Z/p$ in degree $c$.  Thus $BP\langle n \rangle_c$ is atomic in this range.
\end{thm}

The last statement recovers a 1975 theorem of Steve Wilson \cite{wilson II}.

It is natural to wonder if, in the Finiteness Theorem above, the hypothesis that $H^*(X;\Z/p)$ be a finitely presented $A$--module can be weakened to finitely generated. Unfortunately, the answer is no.

\begin{ex} \label{BP ex} The $p$--local Brown--Peterson spectrum $BP$ \cite{ravenel} has mod $p$ cohomology that is cyclic as an $A$--module. But in \propref{BP prop} we give an elementary proof that the Hurewicz map for $BP_0$ has infinite dimensional image in positive degrees.  More deeply, this even holds for $BP_c$ for all $c$: the main theorem of  \cite{wilson II} says that $BP_c$ splits as an infinite product of spaces of the form occurring in \thmref{BPn thm}, and so the mod $p$ Hurewicz image is certainly infinite.
\end{ex}

Details about \thmref{BPn thm} and \exref{BP ex} are in \secref{BP section}, where we continue and compute the mod $p$ Hurewicz maps for all the spaces $BP_c$ and $BP\langle n \rangle_c$.

Still consistent with \exref{BP ex} is the following conjectural strengthening of the Finiteness Theorem.

\begin{conj}  If $H^*(X;\Z/p)$ is a finitely generated $A$--module, then there is an $s$ such that the kernel of $h_*: \pi_*(X) \ra H_*(\Oinfty X;\Z/p)$ contains all elements of Adams filtration at least $s$.
\end{conj}

This is also consistent with the famous Curtis Conjecture \cite{curtis}.

\begin{conj} The kernel of $h: \pi_*^S(S^0) \ra H_*(QS^0;\Z/2)$ consists of all classes except those of odd Hopf or Kervaire invariant.
\end{conj}

\subsection{Applications: Lannes--Zarati theory and  Morava $E$--theory} \label{LZ apps sec}

There are two interesting situations when the augmentation ideal filtration of $R \sm \Sinfty \Oinfty X$, or some localization of this, splits:
\begin{itemize}
\item when $X = \Sinfty Z$ and $R$ is arbitrary.
\item when $X$ and $R$ are arbitrary, and the filtration is localized with respect to the $n$th Morava $K$--theory.
\end{itemize}

This has calculational implications when combined with \thmref{lifting thm adjointed}, as the Hurewicz map then induces a map from an Adams spectral sequence to a trivial spectral sequence.  In the first case, we recover and extend work announced by Lannes and Zarati in the 1980's.  The second case yields a chromatic version of Lannes-Zarati theory.

We now describe how this goes.

When $X = \Sinfty Z$, with $Z$ a based space, our Hurewicz map has the form
$$ h_*: \pi_*^S(Z) \lra  R_*(QZ),$$
where $QZ = \Oinfty \Sinfty Z$, as usual.

In this case, the filtration for $\Sinfty QZ$ is well known to split as $S$--modules.  This was first proved by D.S.Kahn \cite{kahn}; see \cite{kuhn-TAQtowers} for a modern presentation that includes the case when $Z$ is possibly not connected, and multiplicative structure.
The Adams spectral sequence $\{E_r^{s,t}(Z)\}$ thus maps to a trivial spectral sequence $\{\mathcal E_r^{s,t}(Z) \}$ having $\mathcal E_1^{s,t}(Z)  = \bigoplus_{d=p^s}^{p^{s+1}-1} R_{t-s}(D_d(Z))$.  Most of the next theorem is a formal consequence of this.  The refinement of the last statement is a consequence of \addendref{addendum prop} specialized to $K=S^1$.

\begin{thm} \label{space thm}  Let $R$ be as in \thmref{lifting thm}.  Localized at a prime $p$, the refinement of the Hurewicz map
$$ h_*: \pi_*^S(Z) \lra  R_*(QZ)$$
as in \thmref{lifting thm adjointed} induces maps
$$ h_r^{s,t}: E_r^{s,t}(Z) \ra \bigoplus_{d=p^s}^{p^{s+1}-1} R_{t-s}(D_d(Z))$$
all of which factor through a common map
$$ h^{s,t}: Z_1^{s,t}(Z)/B_{\infty}^{s,t}(Z) \ra \bigoplus_{d=p^s}^{p^{s+1}-1} R_{t-s}(D_d(Z)).$$
Furthermore, if $Z \simeq \Sigma W$ for some space $W$, the components of $h^{s,t}$ are zero unless $d=p^s$, and, indeed, the image of $h^{s,t}$ is contained in the image of $\Delta_*: R_{t-s-1}(D_{p^s}(W)) \ra R_{t-s}(D_{p^s}(Z))$.
\end{thm}
Here $Z_r^{s,t}(Z)$ and $B_{r}^{s,t}(Z)$ are the $r$-cycles and $r$-boundaries.

\begin{proof}  Our maps $\pi_*(h_s): \pi_*((\Sinfty Z)(s)) \ra R_*(I^{p^s}(\Sinfty Z))$ induce a map of spectral sequences $\{E_r^{s,t}(Z)\} \xra{\{h_r^{s,t}\}} \{\mathcal E_r^{s,t}(Z)\}$. Since the codomain spectral sequence collapses at $E_1$, it follows that, for fixed $(s,t)$, all the maps $h_r^{s,t}$ factor through a single map $h^{s,t}: Z_1^{s,t}(Z)/B_{\infty}^{s,t}(Z) \ra \mathcal E_1^{s,t}(Z)$.

When $Z = \Sigma W$, \addendref{addendum prop} implies that we have a commutative diagram
\begin{equation*}
\SelectTips{cm}{}
\xymatrix{
Z_1^{s,t-1}(W)/B_{\infty}^{s,t-1}(W) \ar[d]^{\epsilon_{S^1}(W)} \ar[r]^-{h^{s,t-1}} & \mathcal E_1^{s,t-1}(W) \ar[d]^{\epsilon_{S^1}(W)} \ar@{=}[r] &  \bigoplus_{d=p^s}^{p^{s+1}-1} R_{t-s-1}(D_d(W)) \ar[d]^{\bigoplus_d \Delta_*} \\
Z_1^{s,t}(Z)/B_{\infty}^{s,t}(Z) \ar[r]^-{h^{s,t}} & \mathcal E_1^{s,t}(Z)\ar@{=}[r] & \bigoplus_{d=p^s}^{p^{s+1}-1} R_{t-s}(D_d(Z)) }
\end{equation*}
The left vertical map is an isomorphism as the Adams spectral sequence for $\Sigma W$ agrees with the shift of the Adams spectral sequence for $W$.  Meanwhile, since we have localized at $p$, all the components of the right vertical map are zero except when $d=p^s$, as $\Delta: \Sigma D_d(W) \ra D_d(\Sigma W)$ is null in those cases, by standard transfer arguments.  The last assertion of the theorem follows.

\end{proof}

We consider what this theorem says when $R = H\Z/p$.  Write $H^*(Z)$ and $H_*(Z)$ for the mod $p$ cohomology and homology of $Z$, which are respectively left and right $A$--modules.  In this case, there is a natural isomorphism
$$ \mathcal R_s H_*(Z) \simeq \im \Delta_* \subseteq H_*(D_{p^s}(Z)),$$
where $\mathcal R_s$ is a well known functor of the category of locally finite right $A$--modules: roughly put, $\mathcal R_s M$ is the module generated by applying all sequences of Dyer-Lashof operations of length $s$ to a module $M$.  (See \cite{powell, KM13} for modern presentations of the interesting properties of these `Singer functors'.)

We let $\mathcal R M = \bigoplus_{t=0}^{\infty}\mathcal R_t M$.  This has a decreasing filtration with $F_s\mathcal R M = \bigoplus_{t=s}^{\infty} \mathcal R_t M$.  One has a canonical inclusion
$ \mathcal RH_*(Z) \subset H_*(QZ)$
which, when $Z$ is a suspension, identifies with the module of primitives in $H_*(QZ)$.  Thus the Hurewicz map for $QZ$ factors
$$ \pi_*^S(Z) \lra \mathcal RH_*(Z) \subset H_*(QZ).$$

\thmref{space thm} thus tells us the following.

\begin{cor} If a space $Z$ is a suspension, the mod $p$ Hurewicz map
$$ h_*: \pi_*^S(Z) \lra \mathcal RH_*(Z) \subset H_*(QZ)$$
is filtration preserving and
induces commutative diagrams
\begin{equation*}
\SelectTips{cm}{}
\xymatrix{
Ext_{A}^{s,t}(H^*(Z),\Z/p) \ar@{->>}[d] \ar[rrd]^-{h_2^{s,t}} &&   \\
Z^2_{s,t}(Z)/B^{\infty}_{s,t}(Z)\ar[rr]^-{h^{s,t}}  &&    (\mathcal R_sH_*(Z))_{t-s}. \\
E^{\infty}_{s,t}(Z) \ar@{>->}[u] \ar[rru]_-{h_{\infty}^{s,t}} && }
\end{equation*}
\end{cor}

Without calculation, we have thus recovered, and extended to odd primes, 2--primary results of Lannes and Zarati announced in \cite{lz1}. (\cite{lz3} is a partially completed manuscript. See also \cite{L88,lz2}.)  Though we don't show this here, Lannes and Zarati's work suggests that $h_2^{s,t}$ is surely the specialization of an explicit algebraic natural transformation
$$ Ext_{A}^{s,t}(M^{\vee},\Z/p) \ra (\mathcal R_s M)_{t-s},$$
where $M^{\vee}$ denotes the dual of a right $A$--module $M$.

\begin{rem} Though our methods are very different than theirs, the starting point for this paper was the author's realization that these old results could be naturally placed within the context of Goodwillie calculus.
\end{rem}

Now we consider the other case mentioned above when the augmentation ideal filtration splits.
Let $E_n$ denote the $n$th Morava $E$--theory at a fixed prime $p$, and then let $E_*(Y) = \pi_*(L_{K(n)}(E_n \sm Y))$.  There is a natural map
$$ MU \sm Y \ra L_{K(n)}(E_n \sm Y)$$
for all spectra $Y$, and thus a natural map of filtrations
$$ MU \sm I^d(X) \ra L_{K(n)}(E_n \sm I^d(X))$$
for all spectra $X$.

The main result in \cite{kuhn-TAQtowers} implies that the filtration for $L_{K(n)}(\Sinfty \Oinfty X)$ has a natural splitting analogous to the splitting of the filtration for $\Sinfty QZ$ as described above.  It follows that the filtration for $L_{K(n)}(E_n \sm \Sinfty \Oinfty X)$ is similarly split.

We use this, together with \thmref{lifting thm} as follows.  We apply \thmref{lifting thm} to the case when $R = MU$, to get a map of spectral sequences from the Adams-Novikov spectral sequence, which we denote $\{E_r^{*,*}(X)\}$, to the augmentation ideal filtration spectral sequence of $MU_*(\Oinfty X)$.  Then this augmentation ideal spectral sequence maps to the collapsing spectral sequence computing $E_*(\Oinfty X)$, which we denote $\{\mathcal E_r^{*,*}(X)\}$.

Just as before, this splitting and \thmref{lifting thm} show that the map of spectral sequences on the $E_2$--page determines the map on the $E_{\infty}$--page.  In the next theorem, let $E_r^{*,*}(X)$ denote the $r$th page of the Adams--Novikov spectral sequence converging to $\pi_*(X)$, with cycles and boundaries $Z_r^{*,*}(X)$ and $B_r^{*,*}(X)$.

\begin{thm} \label{Morava E theory thm}  Fix a prime $p$ and $n$.  For connective spectra $X$, the $n$th Morava $E$--theory  Hurewicz map
$$ h_*: \pi_*(X) \lra  E_*(\Oinfty X),$$
refined by applying \thmref{lifting thm adjointed} to the first map in the composite
$$ h_*: \pi_*(X) \ra  MU_*(\Oinfty X) \ra E_*(\Oinfty X),$$
induces maps
$$ h_r^{s,t}: E_r^{s,t}(X) \ra \bigoplus_{k=p^s}^{p^{s+1}-1} E_{t-s}(D_k(X))$$
all of which factor through a common map
$$ h^{s,t}: Z_1^{s,t}(X)/B_{\infty}^{s,t}(X) \ra \bigoplus_{k=p^s}^{p^{s+1}-1} E_{t-s}(D_k(X)).$$
Furthermore, if $X$ is 0--connected, the components of $h^{s,t}$ are zero unless $k=p^s$, and, indeed, the image of $h^{s,t}$ is contained in the image of $$\Delta_*: E_{t-s-1}(D_{p^s}(\Sigma^{-1}X)) \ra E_{t-s}(D_{p^s}(X)).$$
\end{thm}

\begin{cor} If $X$ is 0--connected, there is a natural diagram
\begin{equation*}
\SelectTips{cm}{}
\xymatrix{
Ext_{MU_*(MU)}^{s,t}(MU_*, MU_*(X)) \ar@{->>}[d] \ar[rrd]^-{h_2^{s,t}} &&   \\
Z_2^{s,t}(Z)/B_{\infty}^{s,t}(X)  \ar[rr]^-{h^{s,t}} &&    \im \Delta_* \subseteq E_{t-s}(D_{p^s}(X)). \\
E^{\infty}_{s,t}(X) \ar@{>->}[u] \ar[rru]_-{h_{\infty}^{s,t}} && }
\end{equation*}
\end{cor}

\begin{rems} { (a)} If $E_*(X)$ is a finitely generated free $E_*$--module, there is much known about $\im \Delta_* \subset E_{*}(D_{p^s}(X))$; in particular, it is a functor of $E_*(X)$, viewed as an $E_*$--module \cite[Remark 7.4]{rezk}. One might expect that the map $h_2^{*,*}$ has an algebraic description in this case.

{(b)} As $E$ is not bounded below, convergence of the (trivial) tower spectral sequence is both problematic and subtle.  This is studied in detail in \cite{kuhn-TAQtowers}. \end{rems}

\noindent{\bf Acknowledgements} \ The inspiration for our lifting theorem comes from old results of Lannes and Zarati as announced in \cite{lz1}.  Our first thoughts about this, proved only for $p=2$ and modulo the development of the composition product structure on the augmentation ideal filtration, were presented in the September, 2007 algebraic topology workshop in Oberwolfach \cite{kuhnOber07}.  

The writing of the first version of this paper was done during a stay at the Mathematical Sciences Research Institute in Berkeley during the spring 2014 algebra topology program.  The author is grateful for the opportunity to be part of this program which was supported by N.S.F. grant 0932078000, and also for support through N.S.F. grant 0967649.  
The second version, with more developed applications discovered after the MSRI visit, was written while visiting the University of Sheffield during the first half of 2017. The author thanks the Sheffield Mathematics Department for its hospitality.

We thank the referee for a careful reading, leading to a more careful exposition of background preliminaries.

Finally, the author thanks Greg Arone and Luis Pereira for useful conversations about this work.

\section{Background material}\label{background sec}

\subsection{Operads and their categories of modules and algebras}

We recall some definitions related to operads; see, e.g., \cite{fresse} for more detail.

A symmetric sequence in a category $\C$ consists of a sequence
$$X = \{X(0),X(1),X(2),\dots\},$$ where $X(n)$ is an object in $\C$ equipped with an action of the $n$th symmetric group $\Sigma_n$.

If $\C$ is a based symmetric monoidal cocomplete category with product $\sm$, unit $\mathbb I$, and initial/final object $*$, the category of symmetric sequences in $\C$ admits a composition product $\circ$ defined by
\begin{equation*}
 (X \circ Y)(n) = \bigvee_d X(d) \sm_{\Sigma_d} \left(\bigvee_{\phi: \bf n \ra \bf d} Y(n_1(\phi)) \sm \dots \sm Y(n_d(\phi))\right),
\end{equation*}
where ${\bf n} = \{1,\dots,n\}$ and $n_i(\phi)$ is the cardinality of $\phi^{-1}(i)$.
With this product, the category of symmetric sequence in $\C$ is monoidal, with unit object $\mathbb I(1) = (*,\mathbb I,*,*, \dots)$.

An operad $\Oo$ is then a monoid in this category, and one makes sense of left $\Oo$--modules, right $\Oo$--modules, and $\Oo$--bimodules in the usual way.  An $\Oo$--algebra is a left $\Oo$--module $I$ concentrated in level 0: $I(n)=*$ for all $n>0$.
If $M$ is a right $\Oo$--module, and $N$ is a left $\Oo$--module, the symmetric sequence $M \circ_{\Oo} N$ can be defined as the coequalizer in the category of symmetric sequences of the two evident maps
$$ M \circ \Oo \circ N \begin{array}{c} \longrightarrow \\[-.1in] \longrightarrow
\end{array} M \circ N.$$
It is easily checked that if $M$ is an $\Oo$--bimodule and $I$ is an $\Oo$--algebra, then $M \circ_{\Oo} I$ is again an $\Oo$--algebra.

Finally, we let $\Com$ denote the reduced commutative operad:  $\Com(0)=*$ and $\Com(n) = \mathbb I$ for $n>0$.

\subsection{Categories of modules and algebras of symmetric spectra}

Our category of $S$--modules will be the symmetric monoidal category of symmetric spectra of \cite{hss}. If $R$ is a commutative $S$--algebra, we let $R\Mod$ denote the category of $R$--modules,  and  $R\Alg$ denote the category of {\em nonunital} commutative $S$--algebras. We also make use of $R\Mod^{\Sigma}$, the category of reduced symmetric sequences in $R\Mod$: symmetric sequences $M = \{M(n)\}$ satisfying $M(0)=*$.  It is then useful to observe that $R\Alg$ can be regarded as the category of $\Com$--algebras in $R\Mod$.

$R\Mod$, $R\Mod^{\Sigma}$, and $R\Alg$ are based simplicial model categories, with model structures as in \cite{kuhn-pereira,PereiraHHA}, which piggy-back off the positive model structure on $S\Mod$ exploited in \cite{shipley}. (See \cite[Chapter 4]{hirschhorn} for properties of simplicial model categories.)

As part of this structure, the categories are tensored over $sSet_*$, the category of based simplicial sets.

For $K \in sSet_*$ and $M \in R\Mod$, $K \otimes M$ is just the smash product $K \sm M$. Similarly, if $M = \{M(n)\} \in R\Mod^{\Sigma}$, $(K \otimes M)(n) = K \sm M(n)$.

For $I \in R\Alg$, we describe a useful model for $K \otimes I$.  Given $K \in sSet_*$, let $K^{\sm}$ denote the symmetric sequence with $K^{\sm}(n) = K^{\sm n}$.  This is a $\Com$--bimodule in $sSet_*$ as follows.    Given an epimorphism $\phi: {\bf n} \ra {\bf d}$, let $n_i = \#\phi^{-1}(i)$. Then
\begin{itemize}
\item the left $\Com$--module structure on $K^{\sm}$ is induced by the associated concatenation
$ K^{\sm n_1} \sm \dots K^{\sm n_d} \xra{\sim} K^{\sm n}$,
\item the right $\Com$--module structure on $K^{\sm}$ is induced by the associated diagonal $K^{\sm d} \ra K^{\sm n}$.
\end{itemize}

\begin{lem} \label{tensor lem} $K \otimes I = K^{\sm} \circ_{\Com} I$.
\end{lem}
Said in slightly different language, this is observed in \cite[\S 5]{kuhn K tensor R}.  See also \cite[\S 1.4]{PereiraOalg}.

\subsection{Some simplicial functors} \label{simplicial functor sec}

A simplicial functor $F: \A \ra \B$ between two based simplicial model categories is a functor inducing appropriately compatible maps of simplicial sets
$$ F: \Map_{\A}(X,Y) \ra \Map_{\B}(F(X),F(Y))$$
on mapping spaces \cite[Definition 9.8.1]{hirschhorn}.

There are then natural transformations
$$ \epsilon_K: K \otimes F(X) \ra F(K\otimes X)$$
defined as the adjoint to the composite
$$ K \xra{e} \Map_{\A}(X,K \otimes X) \xra{F} \Map_{\B}(F(X),F(K \otimes X)),$$
where $e$ is the unit of the adjoint pair
( \ $ \underline{\text{\hspace{.13in}}}\otimes X, \Map_{\A}(X,\underline{\text{\hspace{.13in}}})$\ ).
These satisfy an obvious unital condition, and a compatibility condition: given $K, L \in sSet_*$, $\epsilon_{K \sm L}$ identifies with the composite $\epsilon_K \circ (K \otimes \epsilon_L)$, using the identification $(K \sm L) \otimes X = K \otimes (L \otimes X)$.

Conversely, if a functor $F$ admits such a compatible family then it is simplicial \cite[Theorem 9.8.5]{hirschhorn}.

The next result follows formally from the definitions.

\begin{lem} \label{nat trans lem} If $\theta: F \ra G$ is a natural transformation between two such simplicial functors, then the diagram
\begin{equation*}
\SelectTips{cm}{}
\xymatrix{
K \otimes F(X) \ar[d]^{\epsilon_K} \ar[rr]^{K \otimes \theta(X)} && K \otimes G(X) \ar[d]^{\epsilon_K}  \\
F(K \otimes X) \ar[rr]^{\theta(K \otimes X)} && G(K \otimes X) }
\end{equation*}
commutes.
\end{lem}

We now run through basic examples important for us.

\begin{ex}  Let $D_d: R\Mod \ra R\Mod$ be the $d$th extended power functor: $D_d(X) = X^{\sm d}_{h\Sigma_d}$.  Then $\epsilon_K: K \sm D_d(X) \ra D_d(K \sm X)$ is induced by the diagonal $\Delta: K \ra K^{\sm d}$.  More precisely, it is obtained from the $\Sigma_d$ equivariant map $K \sm X^{\sm d} \xra{\Delta \sm 1} K^{\sm d} \sm X^{\sm d} \simeq (K \sm X)^{\sm d}$ by passing to homotopy orbits.
\end{ex}

\begin{ex} \label{tensor example}  Given $M,N \in R\Mod^{\Sigma}$, and $K \in sSet_*$, the diagonal maps for $K$ similarly induce maps from
$$K \sm M(d) \sm N(n_1) \sm \dots N(n_d)$$
to
$$M(d) \sm (K \sm N(n_1)) \sm \dots (K \sm N(n_d))$$
for all choices of $d$ and $n_1, \dots, n_d$.
A wedge of such maps induces
$$\epsilon_K: K \sm (M \circ N) \ra M \circ (K \sm N).$$

The obvious compatibility among these maps show that, for a fixed symmetric sequence $M$, the functor
$ M \circ \uspace : R\Mod^{\Sigma} \ra R\Mod^{\Sigma}$
is simplicial.
\end{ex}

\begin{ex}  Let $M$ be a $\Com$--bimodule in $R\Mod$.  The maps in the last example induce natural transformations
$$\epsilon_K: K \otimes (M \circ_{\Com} I) \ra M \circ_{\Com} (K \otimes I)$$
for all $I \in R\Alg$, witnessing that
$ M \circ_{Com} \uspace: R\Alg \ra R\Alg$
is simplicial.
\end{ex}

\begin{ex} \label{bar ex} If $I$ is in $R\Alg$ and $M$ is a $\Com$--bimodule, let $B(M,\Com,I) \in R\Alg$ denote the bar construction: the realization of the simplicial object $B_{\bullet}(M,\Com,I)$ in $R\Mod$ defined by
$$ B_n(M,\Com,I) =  M \circ \overbrace{\Com \circ \cdots \circ \Com}^n \circ I.$$

If $I$ is a cofibrant $R$--algebra, then the natural map
$$B(M,\Com,I) \ra M \circ_{\Com} I$$ is a weak equivalence \cite[Proposition 2.9]{kuhn-pereira}.  As illustrated in \cite{kuhn-pereira}, $B(M,\Com,I)$ is well suited for homotopical analysis.

Again, the maps in \exref{tensor example} induce
natural transformations
$$\epsilon_K: K \otimes B(M,\Com,I) \ra B(M,\Com,K\otimes I)$$
for all $I \in R\Alg$, witnessing that
$B(M, \Com, \uspace): R\Alg \ra R\Alg$
is simplicial.
\end{ex}

\subsection{The augmentation ideal filtration and Andr\'e--Quillen homology} \label{aug ideal properties sec}

We collect needed results from \cite{kuhn-pereira}.

\begin{defn} In the category $R\Mod$, let $\Com^{\geq d}$ be the $\Com$--bimodule with
\begin{equation*}
\Com^{\geq d}(n) =
\begin{cases}
 R & \text{for } n \geq d \\ * & \text{for } n<d.
\end{cases}
\end{equation*}
Then let $ (\uspace)^d: R\Alg \ra R\Alg$, the {\em $d$th power of the ideal functor},  be the simplicial functor defined by letting $I^d = B(\Com^{\geq d},\Com,I)$. The maps of $\Com$-bimodules $\Com^{\geq d+1} \ra \Com^{\geq d}$ then induce natural transformations $$I^{d+1} \ra I^d.$$
\end{defn}

\begin{defn} In the category $R\Mod$, let $R(1)$ be the $\Com$--bimodule with
\begin{equation*}
R(1)(n) =
\begin{cases}
 R & \text{for } n = 1 \\ * & \text{for } n>1.
\end{cases}
\end{equation*}
Then let $ T: R\Alg \ra R\Mod$, the {\em Andr\'e--Quillen homology functor}, be the simplicial functor defined by letting $T(I) = B(R(1),\Com,I)$.  The map of $\Com$-bimodules $\Com \ra R(1)$ then induces a natural transformation $$I \ra T(I).$$
\end{defn}

\begin{thm}\cite{kuhn-pereira}  \label{I properties thm} For all $R$, the functor $$T: R\Alg \ra R\Mod$$ and the natural augmentation ideal filtration in $R\Alg$,
$$ I \la I^2 \la I^3 \la \dots, $$
satisfy the following properties. \\

\noindent{\bf (a)} \ $T$ is the derived left adjoint to the functor $z: R\Mod \ra R\Alg$ giving an $R$--module the trivial $R$--algebra structure; in the associated homotopy categories one has
$$[T(I),M]_{R\Mod} \simeq [I, z(M)]_{R\Alg}.$$

\noindent{\bf (b)} \ There is a homotopy fibration sequence in $R\Alg$
$$ I^{d+1} \ra I^d \ra z(D_dT(I)).$$

\noindent{\bf (c)} \ There are natural composition products
$$ \mu: (I^c)^d \ra I^{cd},$$
compatible as $c$ and $d$ vary, and which are homotopic to the identity if either $c$ or $d$ equals 1.
\end{thm}

\begin{rem} We give some historical context.  The first two properties are well known.  $T(I)$ can be informally viewed as $I/I^2$, and it was first defined, along with its defining property (a), in  \cite{basterra}.  The construction of the augmentation ideal filtration, or more precisely, the associated tower under $I$,
$ I/I^2 \la I/I^3 \la I/I^4 \la \dots$,
goes back to \cite{minasian} and \cite{kuhn-TAQtowers}.  The paper \cite{HH} also constructs this tower, together with the analogue of property (b).    Finally, the construction of natural composition products as in (c), compatible with all these other properties, was the main goal of \cite{kuhn-pereira}: they are induced by the maps of $\Com$--bimodules $$\Com^{\geq d} \circ_{\Com} \Com^{\geq c} \ra \Com^{\geq cd}$$
induced by the operad structure of $\Com$. \\
\end{rem}

Critical to us is that the constructions in the last theorem behave well with respect to change of rings.  To state this, we use the following more precise notation: if $R$ is a commutative $S$--algebra, and $Y$ is an $R$--module then $D_d^R(Y) = (Y^{\sm_R^d})_{h\Sigma_d}$.

\begin{thm}\cite{kuhn-pereira} \label{change of rings thm} Suppose $A \ra B$ is a map of unital commutative $S$--algebras, $I \in A\Alg$, and $J = B \sm_A I \in B\Alg$.  Then there are natural equivalences  $B \sm_A T(I) \xra{\sim} T(J)$ in $B\Mod$ and $B \sm_AI^d \xra{\sim} J^d$ in $B\Alg$ such that the diagram
\begin{equation*}
\SelectTips{cm}{}
\xymatrix{
B \sm_AI^{d+1} \ar[d]^{\wr} \ar[r] & B \sm_AI^d \ar[d]^{\wr} \ar[r] & B \sm_A z(D^A_d(T(I))) \ar[d]^{\wr} \\
J^{d+1} \ar[r] & J^d  \ar[r] & z(D^B_d(T(J)))}
\end{equation*}
commutes in $B\Alg$.
\end{thm}

Finally, our constructions yield connectivity estimates.

\begin{prop}\cite[Prop.2.22]{kuhn-pereira} Let $R$ be a connective commutative $S$--algebra. If $I\in R\Alg$ is $(c-1)$-connected then $I^d$ is $(cd-1)$ connected.
\end{prop}

\subsection{The homotopy category of finitary functors} \label{ho cat section}

With $\C$ one of the categories $S\Mod$ or $S\Mod_{\geq 0}$, and $\D$ one of the categories $R\Mod$ or $R\Alg$, let $\Fun(\C,\D)$ be the category of finitary functors $F: \C \ra \D$.  (We recall that $F$ is finitary if it commutes with directed homotopy colimits \cite[Def.5.10]{goodwillie3}. All functors in this paper have this property.)

Then \cite[Thm.2.14]{biedrondigsII} shows that $\Fun(\C, \D)$ has a  model structure with weak equivalences and fibrations defined objectwise.

For clarity's sake and readability, we will occasionally abuse notation and terminology as we work in these categories.  As examples, we may write `a natural transformation $f(X) \in [F(X),G(X)]_{\D}$' when we mean $(f: F \ra G) \in ho\Fun(\C,\D)$, we may say `$f(X)$ is naturally null' when we mean that $(f: F \ra G) \in ho\Fun(\C,\D)$ is the null element, and we may say, for fixed $F,G$, that `$[F(X),G(X)]_{\D}$ is zero' when we mean that all $f(X) \in [F(X),G(X)]_{\D}$ are naturally null.  We trust the meaning will be clear.

\subsection{The model for $\Sinfty \Oinfty X$}

In this subsection we are working in our homotopy categories of functors. In particular, our weak equivalences are zig-zags of natural transformations that are objectwise weak equivalences.

\begin{thm}  \label{TQ prop} There is a simplicial functor $I: S\Mod \ra S\Alg$, a natural weak equivalence of $S$--modules,
$$ I(X) \simeq \Sinfty \Oinfty X$$
and, for connective $X$, a natural weak equivalence
$$ T(I(X)) \simeq X$$
such that the diagram
\begin{equation*}
\SelectTips{cm}{}
\xymatrix{
I(X) \ar@{-}[d]^{\wr} \ar[r] & T(I(X)) \ar@{-}[d]^{\wr}  \\
\Sinfty \Oinfty X \ar[r]^-{\epsilon} & X }
\end{equation*}
commutes up to homotopy.
\end{thm}

Here the top horizontal map is the canonical map $I \ra T(I)$ specialized to $I = I(X)$.

\begin{rem}  The restriction to connective $X$ here is due to the observation that the connective cover map\footnote{See \secref{lifting lemmas proof sec} for a discussion of the connective cover functor.} $X\langle 0 \rangle \ra X$ induces a weak equivalence after applying $\Oinfty$.  Thus $I(X\langle 0 \rangle) \ra I(X)$ will be a weak equivalence, and it follows that there is a natural weak equivalence $T(I(X)) \simeq X\langle 0 \rangle$ for all $X \in S\Mod$.
\end{rem}

This is a variant of results in \cite{ekmm, basterra-mandell, kuhn-TAQtowers}, with the only difference being that we are working in different a category of spectra, not the category of \cite{ekmm}.

Assuming the first statement for the moment -- the existence of a simplicial functor $I$ together with a natural weak equivalence $I(X) \simeq \Sinfty \Oinfty X$ -- we show that the rest of the theorem follows from this.

Firstly, the natural maps $I \ra T(I)$ induce a weak equivalence
$$ \hocolim_n \Sigma^{-n}(S^n \otimes I) \xra{\sim} T(I).$$
This is a well known consequence of the `linearity' of Andr\'e--Quillen homology \cite{basterra-mandell}, and in our setting can also be seen as being induced by a levelwise weak equivalence of $\Com$--bimodules
$$ \hocolim_n \Sigma^{-n} (S^n)^{\sm} \xra{\sim} S(1).$$

\lemref{tensor lem} and \exref{bar ex} imply that $S^n \otimes I$ is equivalent to an iterated bar construction. It follows that the natural map
$$ \epsilon_{S^n}: S^n \otimes I(X) \ra I(\Sigma^n X)$$
is an equivalence for connective $S$--modules $X$: the right side is a model for the suspension spectrum of $X_n = \Oinfty \Sigma^n X$, and so is the left when $X$ is connective.

Finally, one has that the adjunction maps $\epsilon: \Sinfty \Oinfty X \ra X$ induce a weak equivalence $\displaystyle \hocolim_n \Sigma^{-n} \Sinfty X_n \xra{\sim} X$.

Putting these equivalences together gives equivalences, for connective $X$,
\begin{equation*}
\begin{split}
T(I(X)) \xla{\sim} \hocolim_n \Sigma^{-n}(S^n \otimes I(X)) &
\xra{\sim} \hocolim_n \Sigma^{-n}(I(\Sigma^n X)) \\
  & \xra{\sim} \hocolim_n \Sigma^{-n} \Sinfty X_n \xra{\sim} X
\end{split}
\end{equation*}
compatible with the equivalence $I(X) \simeq \Sinfty \Oinfty X$.

It remains to define the simplicial functor $I: S\Mod \ra S\Alg$.  If we were working with the spectra of \cite{ekmm}, the analogous result is essentially in the literature, with the ideas going back to the beginning of infinite loopspace theory.  Here we follow a parallel path using simplical sets, with the use of the functor $B(\Com, \C_{\infty}^{red}, \text{\hspace{.1in}})$, appearing below, analogous to \cite[p.241, proof of Prop.II.4.5]{ekmm}.

Let $sing$ be the singular simplices functor from topological spaces to the category of simplicial sets, $sSet$.  This has left adjoint $|\text{\hspace{.1in}}|$, the geometric realization functor.

We let $\C_{n}$, $n=1,2, \dots, \infty$, denote the classic little $n$-cubes operad in topological spaces \cite{maygils}. Let $s\C_{n}$ denote $sing$ applied to the operad $\C_{n}$: this is an operad in $sSet$. Then let $s\C_{\infty}\text{-}sSet_*$ be the category of $s\C_{\infty}$ algebras in $sSet_*$.

The functor $\Sinfty_+$ is, as is usual, defined by letting $\Sinfty_+ Z$ be the suspension spectrum functor applied to a simplicial set $Z$ with a disjoint basepoint added.  Applying $\Sinfty_+$ levelwise to $s\C_{\infty}$, we obtain $\Sinfty_+ s\C_{\infty}$, an operad in $S\Mod$.  We let $s\C_{\infty}\Alg$ denote the category of augmented $\Sinfty_+ s\C_{\infty}$ algebras in $S\Mod$.

Finally, let $\Sinfty_+ s\C_{\infty}^{red}$ denote the associated reduced operad: this agrees with $\Sinfty_+ s\C_{\infty}$, except in level 0, where it is $*$, rather than $S$.  Let $s\C_{\infty}^{red}\Alg$ denote the category of $\Sinfty_+ s\C_{\infty}^{red}$ algebras in $S\Mod$.

Our functor $I$ will be a composite of simplicial functors:
\begin{multline*}
S\Mod
\xra{{\widetilde \Omega}^{\infty} \circ (\text{\hspace{.1in}})^f}
     s\C_{\infty}\text{-}sSet_*
     \xra{\Sinfty_+ \circ (\text{\hspace{.1in}})^c} s\C_{\infty}\Alg \\ \xra{\text{Fib(\hspace{.1in})} \circ (\text{\hspace{.1in}})^f} s\C_{\infty}^{red}\Alg
     \xra{B(\Com, s\C_{\infty}^{red}, \text{\hspace{.1in}})} S\Alg.
\end{multline*}

The functors $(\text{\hspace{.1in}})^f$ and $(\text{\hspace{.1in}})^c$ are simplicial functorial fibrant and cofibrant replacement functors.  These exist by \cite[Proposition 6.3]{rss}.

The functor ${\widetilde \Omega}^{\infty}$ is given by
$$\displaystyle {\widetilde \Omega}^{\infty}(X) = \colim_n \Omega^n sing |X_n| = \colim_n sing \ \Omega^n_{Top}|X_n|.$$  Here $\Omega^n_{Top}$ is the $n$--fold loop functor on based topological spaces.  In the standard way, $\Omega^n_{Top}|X_n|$ is equipped with a $\C_{n}$ operad action \cite{maygils}, and it follows that $\Omega^n sing |X_n|$ is an $s\C_n$--algebra in $sSet_*$. (See also \cite[\S 5.2.6]{lurie} for a more detailed discussion.)

If $A$ is an object in  $s\C_{\infty}\Alg$, $\text{Fib}(A)$ is the fiber of the augmentation $A \ra S$.

The functor $B(\Com, s\C_{\infty}^{red}, \text{\hspace{.1in}})$ sends $J \in s\C_{\infty}^{red}\Alg$ to $B(\Com, \Sinfty_+ s\C_{\infty}^{red}, J)$.  Here the map of operads $\Sinfty_+ s\C_{\infty}^{red} \ra \Com$ makes $\Com$ into a left $\Com$--module and a right $\Sinfty_+ s\C_{\infty}^{red}$--module.

Let $J(\text{\hspace{.1in}})$ be the composite of all the functors in the composite except for the last one. By construction, there is a natural weak equivalence of $S$--modules $J(X) \simeq \Sinfty \Oinfty X$.  As $\Sinfty_+ s\C_{\infty}^{red} \ra \Com$ is a levelwise equivalence of $\C_{\infty}$--bimodules, \cite[Theorem 2.11(a)]{kuhn-pereira} applies to show that there are weak equivalences
$$ J(X) \xla{\sim} B(\Sinfty_+ s\C_{\infty}^{red}, \Sinfty_+ s\C_{\infty}^{red}, J(X)) \xra{\sim} B(\Com, \Sinfty_+ s\C_{\infty}^{red}, J(X)) = I(X).$$

\section{The proof of \thmref{lifting thm}}\label{lifting thm proof sec}

\subsection{The proof, modulo two lifting lemmas}

We now have the ingredients for the proof of \thmref{lifting thm}.  In the following discussion, assume that $R/S$ is 0--connected and that $X$ is connective.

We will see that \thmref{TQ prop}, together with properties (a) and (b) of \thmref{I properties thm}, lead to our first lifting lemma.

\begin{lem} \label{lifting lemma 1}  There is a natural lifting in $ho\Fun(S\Mod_{\geq 0}, R\Alg)$:
\begin{equation*}
\SelectTips{cm}{}
\xymatrix{
& R \sm I^2(X) \ar[d]  \\
R \sm I(X(1)) \ar@{-->}[ur] \ar[r]& R \sm I(X). }
\end{equation*}
\end{lem}
The proof is postponed to \secref{lifting lemmas proof sec}.

When $p>2$, one lifts further using a slightly delicate lifting lemma which we will prove in \secref{lifting lemmas proof sec} using Goodwillie homotopy calculus.

\begin{lem} \label{lifting lemma 2} Localized away from $(p-1)!$,  any natural transformation $f: R \sm I(X(1)) \ra R \sm I^2(X)$ in $ho\Fun(S\Mod_{\geq 0}, S\Alg)$ lifts uniquely:
\begin{equation*}
\SelectTips{cm}{}
\xymatrix{
 & R \sm I^p(X) \ar[d]  \\
R \sm I(X(1)) \ar[r]^f \ar@{-->}[ur]^{\tilde f}& R \sm I^2(X). }
\end{equation*}
\end{lem}

Assuming these lemmas, one can prove our main theorem.
\begin{proof}[Proof of \thmref{lifting thm}]  Note that, given $I \in S\Alg$ and $J \in R\Alg$, $S$--algebra maps $I \ra J$ correspond to $R$--algebra maps $R \sm I \ra J$, so \thmref{lifting thm} can be viewed as asserting that there are compatible natural transformations in $ho\Fun(S\Mod_{\geq 0}, R\Alg)$,
$$ \tilde h_s: R \sm I(X(s)) \ra R \sm I^{p^s}(X) = (R \sm I(X))^{p^s},$$
with $\tilde h_0$ being the identity.

The two lifting lemmas combine to define $\tilde h_1 \in ho\Fun(S\Mod_{\geq 0}, S\Alg)$ making the following commute:
\begin{equation*}
\SelectTips{cm}{}
\xymatrix{
 & (R \sm I(X))^p \ar[d]  \\
R \sm I(X(1)) \ar[r] \ar@{-->}[ur]^{\tilde h_1}& R \sm I(X). }
\end{equation*}

Now assume by induction that $\tilde h_{s-1}$ has been defined.  Then let $\tilde h_s$ be the composite
$$ R \sm I(X(s)) \xra{\tilde h_1} (R \sm I(X(s-1)))^p \xra{(\tilde h_{s-1})^p} ((R \sm I(X))^{p^{s-1}})^p \xra{\mu} (R \sm I(X))^{p^s}.$$

The commutative diagram
\begin{equation*}
\SelectTips{cm}{}
\xymatrix{
R \sm I(X(s-1)) \ar[ddd]^{\tilde h_{s-1}} &   & R \sm I(X(s)) \ar[ddd]^{\tilde h_s} \ar[ll] \ar[dl]_{\tilde h_1} \\
&  (R \sm I(X(s-1)))^p \ar[ul] \ar[d]^{(\tilde h_{s-1})^p} &   \\
&  ((R \sm I(X))^{p^{s-1}})^p \ar[dl] \ar[dr]^{\mu} &   \\
(R \sm I(X))^{p^{s-1}} &   & (R \sm I(X))^{p^s} \ar[ll] }
\end{equation*}
shows $\tilde h_s$ has the needed compatibility with $\tilde h_{s-1}$.  Note that the commutation of the bottom triangle is an example of the compatibility stated in  \thmref{I properties thm}(c).

\end{proof}

\subsection{Proof of the lifting lemmas} \label{lifting lemmas proof sec}

Here we are using our terminology and notational conventions as described in \secref{ho cat section}.

\begin{proof}[Proof of \lemref{lifting lemma 1}]

There is a natural homotopy fibration sequence in $R\Mod$:
$$ R\sm I^2(X) \ra R\sm I(X) \ra z(R \sm X).$$
Here we have used properties (b) and (c) of \thmref{I properties thm}, and have identified $T(I(X))$ with $X$ as in \thmref{TQ prop}.  Also note that $X(1) \simeq S(1) \sm X$, where $S(1)$ is the homotopy fiber of $S \ra R$.

A lifting in $ho\Fun(S\Mod_{\geq 0}, R\Alg)$,
\begin{equation*}
\SelectTips{cm}{}
\xymatrix{
& R \sm I^2(X) \ar[d]  \\
R \sm I(X(1)) \ar@{-->}[ur] \ar[r]^{i(X)} & R \sm I(X) \ar[d]^{p(X)} \\
& z(R \sm X), }
\end{equation*}
will exist if we show that there is a natural null homotopy in $R\Alg$ of
$$p(X) \circ i(X): R\sm I(S(1) \sm X) \ra z(R \sm X).$$
By property (a) of \thmref{I properties thm}, we can equivalently, show that
$$T(R\sm I(S(1) \sm X)) \ra R \sm X$$ is naturally null in $R\Mod$.  Using \thmref{TQ prop} again, together with the change of rings theorem \thmref{change of rings thm}, it follows that we just need to show that
$$S(1) \sm X \ra R \sm X$$
is naturally null in $S\Mod$.  But this is tautologically true, as smashing a fixed null homotopy of $S(1) \ra S \ra R$ in $S\Mod$ with $X$ does the job.
\end{proof}

\begin{proof}[Proof of \lemref{lifting lemma 2}]
Reasoning as in the proof of the previous lemma, the obstructions to finding a natural lifting $\tilde f$ in $ho\Fun(S\Mod_{\geq 0}, R\Alg)$ in the diagram
\begin{equation*}
\SelectTips{cm}{}
\xymatrix{
 & R \sm I^p(X) \ar[d]  &\\
  & \vdots \ar[d] & \\
   & R \sm I^3(X) \ar[d]  \ar[r]& z(R \sm D_3(X)) \\
R \sm I(X(1)) \ar[r]^f \ar@{-->}[uuur]^{\tilde f}& R \sm I^2(X) \ar[r] & z(R \sm D_2(X)) }
\end{equation*}
will be natural elements $o_k(X)$, for $k=2, \dots, p-1$, in
\begin{equation*}
\begin{split}
[R \sm S(1) \sm X, z(R \sm D_kX)]_{R\Alg} & \simeq [R \sm S(1) \sm X, R \sm D_kX]_{R\Mod} \\
  & \simeq [ S(1) \sm X, R \sm D_kX]_{S\Mod}.
\end{split}
\end{equation*}

The existence of the lifting as in the lemma will follow if we show that $o_k(X) \in [ S(1) \sm X, R \sm D_kX]_{S\Mod}$ must be zero, after localizing away from $k!$.  Indeed, we will show that the whole homotopy group of natural transformations $[S(1) \sm X, R[\frac{1}{k!}] \sm D_kX]_{S\Mod}$ is zero. Similarly, we will show that $[\Sigma S(1) \sm X, R[\frac{1}{k!}] \sm D_kX]_{S\Mod} = 0$, proving the uniqueness of the lifting.

As a first step towards showing this, we note that transfer arguments show that $X^{\sm k} \ra D_kX$ is naturally split epic, localized away from $k!$.  Thus we have epimorphisms
$$ [S(1) \sm X, R[\frac{1}{k!}] \sm X^{\sm k}]_{S\Mod} \ra [S(1) \sm X, R[\frac{1}{k!}] \sm D_kX]_{S\Mod},$$
and
$$ [\Sigma S(1) \sm X, R[\frac{1}{k!}] \sm X^{\sm k}]_{S\Mod} \ra [\Sigma S(1) \sm X, R[\frac{1}{k!}] \sm D_kX]_{S\Mod},$$
so it suffices to show that the domains of these two homomorphisms are zero for $k\geq 2$.

Readers who are familiar with Goodwillie calculus may possibly think the proof is done, as it is a standard fact that, fixing $S$--modules $M$ and $N$ and $k \geq 2$, there are no essential natural transformations of the form $M \sm X \ra N \sm X^{\sm k}$ with $X \in S\Mod$.  However, in our above analysis, we were assuming $X \in S\Mod_{\geq 0}$, the category of {\em connective} $S$--modules, so we need a more careful argument.

We use the notational convention\footnote{We apologize for the clash of notation, which is not consistent with our previous use of $BP\langle n \rangle$, but trust that the reader will follow.} that $X\langle c \rangle$ denotes the $(c-1)$--connected cover of $X$.  We recall that $X\langle c \rangle$ can be constructed as a colocalization of $S\Mod$, so there is a functor $S\Mod \ra S\Mod$, $X \mapsto X\langle c \rangle$, equipped with a natural transformation $X\langle c \rangle \ra X$.

We observe that the homotopy group $[M \sm X, N \sm X^{\sm k}]_{S\Mod}$, viewed as a set of morphisms in $ho \Fun(S\Mod_{\geq 0}, S\Mod)$, will equal the homotopy group $[M \sm X\langle 0 \rangle, N \sm X\langle 0 \rangle^{\sm k}]_{S\Mod}$, viewed as a set of morphisms in $ho \Fun(S\Mod, S\Mod)$.

Thus the proof of \lemref{lifting lemma 2} will follow from the following lemma.

\begin{lem} \label{delicate lemma}  Let $M, N$ be fixed $S$--modules.  For $k\geq 2$, all natural transformations of the form
$$ \Theta(X): M \sm X\langle 0\rangle \ra N \sm X\langle 0\rangle^{\sm k}$$
are naturally null.
\end{lem}

The proof of this lemma is a bit delicate, and has two rather distinct steps.  We emphasize again we are working in $ho \Fun(S\Mod,S\Mod)$. We also assume some familiarity with basic aspects of Goodwillie homotopy calculus, as in \cite{goodwillie3, kuhn kinosaki}.

The first step is a reduction to \lemref{delicate lemma 2}, below.

Let $F(X) =  M \sm X\langle 0\rangle$ and $G(X) = N \sm X\langle 0\rangle^{\sm k}$.  These are simplicial functors, and we let $\Delta: \Sigma F(X) \ra F(\Sigma X)$ and  $\Delta: \Sigma G(X) \ra G(\Sigma X)$ be the natural maps, as in \secref{simplicial functor sec}.

Since $(\Sigma^{-1}X)\langle 0 \rangle = \Sigma^{-1}(X\langle 1 \rangle)$,  the natural commutative square
\begin{equation*}
\xymatrix{
\Sigma F(\Sigma^{-1}X) \ar[d]^{\Delta} \ar[rr]^-{\Theta(\Sigma^{-1}X)} && \Sigma G(\Sigma^{-1}X) \ar[d]^{\Delta}  \\
F(X) \ar[rr]^{\Theta(X)} && G(X) }
\end{equation*}
rewrites as
\begin{equation*}
\xymatrix{
M \sm X\langle 1\rangle \ar[d] \ar[rr]^-{\Theta(\Sigma^{-1}X)} && \Sigma N \sm (\Sigma^{-1} X\langle 1\rangle)^{\sm k} \ar[d]  \\
M \sm X\langle 0\rangle \ar[rr]^{\Theta(X)} && N \sm X\langle 0\rangle^{\sm k}. }
\end{equation*}
The right vertical map factors as the composite
$$ \Sigma N \sm (\Sigma^{-1} X\langle 1\rangle)^{\sm k} \ra \Sigma N \sm (\Sigma^{-1} X\langle 0\rangle)^{\sm k} \xra{\Delta} N \sm X\langle 0\rangle^{\sm k},$$
which is naturally null, as the second map in this composite is the identity map for $N \sm (\Sigma^{-1} X\langle 0\rangle)^{\sm k}$ smashed with the null map $S^1 \xra{\Delta} S^k$. (Note that $\Delta$ is null since $k\geq 2$.)

We conclude that the composite
$$ M \sm X\langle 1\rangle \ra M \sm X\langle 0 \rangle \xra{\Theta(X)} N \sm X\langle 0\rangle^{\sm k}$$
is naturally null, and thus $\Theta(X)$ naturally factors through a natural transformation of the form
$$ M \sm H\pi_0(X) \xra{\Psi(X)} N \sm X\langle 0\rangle^{\sm k}.$$

\lemref{delicate lemma} will thus follow from the next lemma.

\begin{lem} \label{delicate lemma 2}  Let $M, N$ be fixed $S$--modules.  For $k\geq 1$, all natural transformations of the form
$$ \Psi(X): M \sm H\pi_0(X) \ra N \sm X\langle 0\rangle^{\sm k}$$
are naturally null.
\end{lem}

To prove this, we let $H(X) = M \sm H\pi_0(X)$ and, as before, let $G(X) = N \sm X\langle 0\rangle^{\sm k}$, so that our natural transformation has the form
$$ \Psi(X): H(X) \ra G(X).$$

We consider the $k$th Goodwillie-Taylor approximations to $G$ and $H$.  As observed by Goodwillie \cite[Remark 1.1]{goodwillie3}, the explicit construction of the Taylor towers shows that, if a natural transformation $A(X) \ra B(X)$ of finitary functors is an equivalence on highly connected $S$--modules $X$, then $p_kA(X) \ra p_kB^{\prime}(X)$ will be an equivalence for all $X$.

Applying this observation to $N \sm X\langle 0\rangle^{\sm k} \ra N \sm X^{\sm k}$, we deduce that $(p_kG)(X) \simeq N \sm X^{\sm k}$.

Applying the observation to $M \sm H\pi_0(X) \ra *$, we see that $(p_kH)(X) \simeq *$.

Now $\Psi(X)$ fits into the diagram
\begin{equation*}
\xymatrix{
H(X) \ar[d]^{\Psi(X)} &&  H(X\langle 0\rangle) \ar[ll]_-{\sim} \ar[d]^{\Psi(X\langle 0 \rangle)} \ar[rr] && (p_kH)(X\langle 0\rangle)\simeq * \ar[d]^{(p_k \Psi)(X\langle 0 \rangle)}  \\
G(X) && G(X\langle 0\rangle) \ar[ll]_-{\sim} \ar[rr]^-{\sim} && (p_kG)(X\langle 0\rangle), }
\end{equation*}
and we conclude that $\Psi(X)$ is null, proving the \lemref{delicate lemma 2}.
\end{proof}

\section{The Connectivity Theorem for mod p homology} \label{mod p applications sec}

Let $\{E_r^{*,*}(X)\}$ denote the classic mod $p$ Adams spectral sequence for a spectrum $X$.  A permanent cycle in $E_2^{s,t}(X) = \Ext_{A}^{s,t}(H^*(X;\Z/p), \Z/p)$ represents an element $\alpha \in \pi_{t-s}(X)$ having Adams filtration $s$.

The Connectivity Theorem can be restated as saying that, if $X$ is $(c-1)$--connected, then $h: \pi_*(X) \ra H_*(\Oinfty X;\Z/p)$ satisfies:
\begin{equation} \label{log condition} h(\alpha) = 0 \text{ unless } t-s\geq cp^s.
\end{equation}
Thus the kernel of $h$ includes everything represented by $(s,t)$ with $s>\log_p(t-s) - \log_p(c)$, i.e. everything above a logarithmic curve in the usual graphic depiction of the Adams spectral sequence.

\subsection{Proof of the Finiteness Theorem} \label{finiteness section} We are assuming that $X$ is a bounded below spectrum of finite type, and that  $H^*(X;\Z/p)$ is a finitely presented module over the mod $p$ Steenrod algebra $A$.  We wish to show that $h: \pi_{*}(X) \ra H_*(\Oinfty X;\Z/p)$ has finite dimensional image in positive dimensions.

As a first reduction, we show that we can assume that $X$ is 0--connected.  Note that the basepoint component of $\Oinfty X$ equals $\Oinfty(X\langle 1 \rangle)$, and thus the positive dimensional Hurewicz image for $X$ equals the image of
$h: \pi_{*}(X\langle 1 \rangle) \ra H_*(\Oinfty(X\langle 1 \rangle);\Z/p)$.  The next lemma now shows that $X\langle 1 \rangle$ also will have finitely presented mod p cohomology.

\begin{lem}   Suppose that $X$ is a bounded below spectrum of finite type.  If $H^*(X;\Z/p)$ is a finitely presented $A$--module, so is $H^*(X\langle c\rangle;\Z/p)$ for any $c$.
\end{lem}
\begin{proof} We prove this by induction on $c$.  The statement is obviously true for any $c$ below the connectivity of $X$.   Thus we assume $H^*(X\langle c\rangle;\Z/p)$ is finitely presented, and need to show that $H^*(X\langle c+1 \rangle;\Z/p)$ is finitely presented.

J. Cohen \cite{joel cohen} observed that $A$ is a coherent graded ring, and that this has the following consequence: if $X \ra Y \ra Z$ is a fibration sequence of spectra, and any two of these have finitely presented mod $p$ cohomology, then so does the third.

We apply this to the fibration $X\langle c+1\rangle \ra X\langle c\rangle \ra H\pi_c(X)$ to deduce that $H^*(X\langle c+1 \rangle;\Z/p)$ is finitely presented.  Here we use that  $\pi_c(X)$ is a finitely generated abelian group since $X$ is of finite type, and that, if $B$ is a finitely generated abelian group, then $H^*(HB;\Z/p)$ is a finitely presented $A$--module. (This last fact follows from the standard calculations: $H^*(H\Z/p;\Z/p)=A$, $H^*(H\Z;\Z/p)=A/A\beta$, and $H^*(H\Z/{p^k};\Z/p)=A/A\beta \oplus A/A\beta$ for $k>1$.)
\end{proof}

Thus now assume that for some $c>0$,  $X$ is a $(c-1)$--connected spectrum with finitely presented mod $p$ cohomology.

We need an algebraic lemma. To state this, recall \cite[p.103]{ravenel} that the Steenrod algebra $A$ is the union of the finite dimensional sub--Hopf algebras $A(n)$, where, when $p=2$, $A(n)$ is generated by $Sq^{2^j}$ with $j\leq n$, and when $p$ is odd, $A(n)$ is generated by $\beta$ and $P^{p^j}$ with $j <n$.

\begin{lem} \label{vanishing lemma} (a) If $M$ is a finitely generated $A$--module, then each group $\Ext_{A}^{s,t}(M, \Z/p)$ is finite.  \\

\noindent{\bf (b)} If an $A$--module $M$ has a finite number of relations in an $A$--module  presentation, then, for large enough $n$, there is an $A(n)$--module $N$ such that $M \simeq A \otimes_{A(n)}N$. \\

\noindent (c) Suppose $M$ is an $A$--module of the form $M = A \otimes_{B}N$, with $B$ a finite dimensional sub--Hopf algebra of $A$ having top degree $d(B)$, and $N$ a finitely generated $B$--module with generators $x_1, \dots, x_k$. If $a = \max\{ |x_i|, i=1, \dots, k\}$ and $b=d(B)-1$, then $\Ext_{A}^{s,t}(M, \Z/p) = 0$ unless $t-s \leq a+bs$.
\end{lem}

Assuming this, the proof of the Finiteness Theorem can be finished.  The spectrum $X$ is $(c-1)$--connected for some $c>0$, and has finitely presented mod $p$ cohomology.  The last lemma applies, as does the Connectivity Theorem, repackaged as (\ref{log condition}).  But for fixed $a$ and $b$, and $c>0$, only a finite number of pairs $(s,t)$ satisfy $cp^s \leq t-s < a+bs$, and the size of the image of $h_*$ will be bounded by the size of the sum of the corresponding $\Ext$--groups.

\begin{proof}[Proof of \lemref{vanishing lemma}] Part (a) of the lemma is standard, using that $A$ is a connected algebra of finite type.

To see that statement (b) holds, suppose that $M$ has generators $x_i$, and generating relations $r_1, \dots r_l \in \bigoplus_i \Sigma^{|x_i|}A$.  For big enough $n$, these relations will all be in $\bigoplus_i \Sigma^{|x_i|}A(n)$.  If we let $N = \bigoplus_i \Sigma^{|x_i|}A(n)/A(n)\{r_1, \dots, r_l\}$, then
$$M \simeq A \otimes_{A(n)} N.$$

Now suppose we are in the situation of statement (c).  By \cite[Thm.4.4]{milnor moore} $A$ will be a free $B$--module. We deduce that
$$\Ext_{A}^{s,t}(M, \Z/p) = \Ext_{B}^{s,t}(N, \Z/p) = H^s(\Hom_{B}(P(*), \Sigma^t \Z/p))$$
where $\dots \ra P(2) \ra P(1) \ra P(0) \ra N \ra 0$ is any $B$--projective resolution of $N$.

Recall that $d(B)$ is the top nonzero degree of $B$, and $a = \max\{ |x_i|, i=1, \dots, k\}$.  It is straightforward to show by induction on $s$, that in a minimal projective resolution of the $B$--module $N$, the top nonzero degree of $P(s)$ will be at most $a+d(B)s$.

It follows that, for this resolution,  $\Hom_{B}(P(s), \Sigma^t N) = 0$ if $t>a+d(B)s$, so that
$\Ext_{A}^{s,t}(M, \Z/p) = 0$ unless $t\leq a+d(B)s$, and so, finally, $\Ext_{A}^{s,t}(M, \Z/p) = 0$ unless $t-s\leq a+(d(B)-1)s$.
\end{proof}

\subsection{Proof of the Atomicity Theorem} \label{atomic section}

Suppose $X$ is a connective spectrum and $H^*(X;\Z/p)$ is a finitely presented $A$--module. In terms of the Adams spectral sequence for $X$, the image of the stable mod $p$ Hurewicz map is precisely the graded group $E_{\infty}^{0,*}(X)$. Thus the theorem is asserting that every element in $\pi_{t-s}(X)$ represented by an element in $E_{\infty}^{s,t}(X)$ with $s>0$ must be in the kernel of the mod $p$ Hurewicz map for $\Oinfty \Sigma^c X$ when $c$ is large enough.

By \lemref{vanishing lemma}, there will exist $a$ and $b$ such that $E_{2}^{s,t}(X) \neq 0$ only if $t-s \leq a + bs$.  Since $E_{r}^{s,t}(\Sigma^c X) = E_{r}^{s,t-c}(X)$, we conclude that $E_{\infty}^{s,t}(\Sigma^c X) \neq 0$ only if $t-s \leq c+ a + bs$.  But \thmref{conn thm} implies that if an element in $\pi_{t-s}(X)$ is represented in $E_{\infty}^{s,t}(\Sigma^c X)$ and has nonzero Hurewicz image, then $cp^s \leq t-s$.  But, for fixed $a$ and $b$, if $c$ is large enough the only solutions to $cp^s \leq t-s \leq c+a+bs$ will have $s=0$.

\subsection{The mod 2 Hurewicz map of the connected covers of $BO$} \label{bo section}

In the next few subsections we use the Connectivity Theorem, \thmref{conn thm}, to work out some mod $p$ examples.  All of these examples will also illustrate the Finiteness Theorem, \thmref{finiteness thm}.

We begin with the calculation of
$$ h_*: \pi_*(ko\langle c \rangle) \ra H_*(BO\langle c \rangle;\Z/2)$$
for all $c > 0$.  As computing this is equivalent to computing
$$ H_n(BO\langle n \rangle;\Z/2) \ra H_n(BO\langle c \rangle;\Z/2)$$
for all $n \geq c$, this result is surely accessible by a close reading of \cite{stong}.  However, what we do here involves much less calculation.

It is a standard calculation that the connective real $K$--theory spectrum $ko$ has mod 2 cohomology as a module over the Steenrod algebra given by
$$H^*(ko;\Z/2) = A/A\{Sq^1,Sq^2\} = A \otimes_{A(1)} \Z/2$$
Here $A(1)$ is the sub(Hopf)--algebra generated by $Sq^1$ and $Sq^2$.  This leads to the standard picture \cite[p.66]{ravenel} of the $E_2$--term of the Adams spectral sequence, pictured here as Figure 1, and the conclusion that $E_2 = E_{\infty}$.  (Each $\bullet$ stands for $\Z/2$, as usual.)

\begin{figure}
$$\begin{array}{cc|cccccccccccccc}
 && \vdots& & &  &\vdots & & & &\vdots & &&&\vdots&\\
&8 & \bullet& & & & \bullet & & & &\bullet & &&&\bullet&\\
&7 & \bullet& & &  & \bullet &  &  &  &\bullet  & &&&\bullet&\\
&6 & \bullet& & & & \bullet & & & &\bullet & &\bullet&&&\\
&5 & \bullet& &  &  & \bullet & &  & &\bullet &\bullet &&&&\\
s&4 &\bullet & & & & \bullet & & & &\bullet & &&&&\\
&3 & \bullet& &  &  & \bullet &  &  & & & &&&&\\
&2 &\bullet & & \bullet && & & & & & &&&&\\
&1 & \bullet  & \bullet &  &  &  &  & & & & &&&&\\
&0 & \bullet  &  &  &  &  &  & & & & &&&&\\
\cline{3-16}
\multicolumn{1}{c}{}
& \multicolumn{1}{c}{}
& \multicolumn{1}{c}{0}
& \multicolumn{1}{c}{1}
& \multicolumn{1}{c}{2}
& \multicolumn{1}{c}{3}
& \multicolumn{1}{c}{4}
& \multicolumn{1}{c}{5}
& \multicolumn{1}{c}{6}
& \multicolumn{1}{c}{7}
& \multicolumn{1}{c}{8}
& \multicolumn{1}{c}{9}
& \multicolumn{1}{c}{10}
& \multicolumn{1}{c}{11}
& \multicolumn{1}{c}{12}
& \multicolumn{1}{c}{13}\\
\multicolumn{1}{c}{}
& \multicolumn{1}{c}{}
& \multicolumn{1}{c}{}
& \multicolumn{1}{c}{}
& \multicolumn{1}{c}{}
& \multicolumn{1}{c}{}
& \multicolumn{1}{c}{}
& \multicolumn{1}{c}{}
& \multicolumn{1}{c}{}
& \multicolumn{1}{c}{t-s}
& \multicolumn{1}{c}{}
& \multicolumn{1}{c}{}
& \multicolumn{1}{c}{}
& \multicolumn{1}{c}{}
& \multicolumn{1}{c}{}
& \multicolumn{1}{c}{}
\end{array}$$
\caption{$\Ext_{A}^{s,t}(H^*(ko;Z/2), \Z/2)$}
\end{figure}

It follows quite easily that, when $c\equiv 0,1,2,4 \mod 8$, the corresponding $\Ext$ chart for $ko\langle c\rangle$ looks identical, except that columns with $t-s<c$ are now zero, and Adams filtration $s$ has been decreased so that the bottom class, with degree $t-s=c$, has Adams filtration 0.  Again $E_2 = E_{\infty}$.

For example, the chart for $bo=ko\langle 1 \rangle$ is pictured in Figure 2.  Only the four classes labeled with $\textcolor{cyan}{\bullet}$ satisfy the inequality $t-s \geq 2^s$.  \thmref{conn thm} thus tells us that
$$ h_*: \pi_n(bo) \ra H_n(BO;\Z/2)$$
can only possibly be nonzero when $n=1,2,4,8$.

As suggested in the introduction, this certainly happens: $h_*$ being nonzero in degree $n$ is equivalent to finding a real vector bundle $\xi$ over $S^n$ with top Steifel--Whitney class $w_n(\xi) \neq 0$, and the existence of $\R,\Cx,\Ham$, and $\Oct$ allows one to construct the bundles.

We have thus reproved Milnor's theorem.  It is interesting to see what we have used about $BO$: just that that it is the 0th space of a spectrum $ko$ with $H^*(ko;\Z/2) = A \otimes_{A(1)}\Z/2$.

We can be even more precise about the Hurewicz image in this case.  The image of $h_*$ lands in the primitives of $H_*(BO;\Z/2)$. Dual to the fact that $H^*(BO;\Z/2)$ has one generator in each degree, there is one nonzero primitive in each degree of $H_*(BO;\Z/2)$. If $x \in H_1(BO;\Z/2)$ is the nonzero class, it is clearly primitive, and thus so are $x^{2^k}$ for all $k$.  As $x$ is well known to not be nilpotent, one concludes that $\im h_* = \langle x,x^2,x^4,x^8\rangle$.

Similarly \thmref{conn thm} allows us to deduce the other parts of \thmref{bo thm}:
$$ \im\{ h_*: \pi_*(bso) \ra H_*(BSO;\Z/2)\} = \langle y,y^2,y^4 \rangle \text{ with } |y|=2$$
(no surprise, as $BO = BSO \times B\Z/2$ as spaces),
$$ \im\{ h_*: \pi_*(bspin) \ra H_*(BSpin;\Z/2)\} = \langle z,z^2 \rangle \text{ with } |z|=4,$$
and, for all $c\geq 8$, with $c\equiv 0,1,2,4 \mod 8$,
$$ \im\{ h_*: \pi_*(ko\langle c \rangle) \ra H_*(BO\langle c \rangle;\Z/2)\} = \langle v \rangle \text{ with } |v|=c.$$

\begin{figure}
$$\begin{array}{cc|cccccccccccccc}
 && & & &  &\vdots & & & &\vdots & &&&\vdots&\\

&8 & & & &  &\bullet & & & &\bullet & &&&\bullet&\\
&7 & & & & & \bullet & & & &\bullet & &&&\bullet&\\
&6 & & & &  & \bullet &  &  &  &\bullet  & &&&\bullet&\\
&5 & & & & & \bullet & & & &\bullet & &\bullet&&&\\
s&4 & & &  &  & \bullet & &  & &\bullet &\bullet &&&&\\
&3 & & & & & \bullet & & & &\textcolor{cyan}{\bullet} & &&&&\\
&2 & & &  &  &\textcolor{cyan}{\bullet} &  &  & & & &&&&\\
&1 & & & \textcolor{cyan}{\bullet} && & & & & & &&&&\\
&0 &   & \textcolor{cyan}{\bullet} &  &  &  &  & & & & &&&&\\
\cline{3-16}
\multicolumn{1}{c}{}
& \multicolumn{1}{c}{}
& \multicolumn{1}{c}{0}
& \multicolumn{1}{c}{1}
& \multicolumn{1}{c}{2}
& \multicolumn{1}{c}{3}
& \multicolumn{1}{c}{4}
& \multicolumn{1}{c}{5}
& \multicolumn{1}{c}{6}
& \multicolumn{1}{c}{7}
& \multicolumn{1}{c}{8}
& \multicolumn{1}{c}{9}
& \multicolumn{1}{c}{10}
& \multicolumn{1}{c}{11}
& \multicolumn{1}{c}{12}
& \multicolumn{1}{c}{13}\\
\multicolumn{1}{c}{}
& \multicolumn{1}{c}{}
& \multicolumn{1}{c}{}
& \multicolumn{1}{c}{}
& \multicolumn{1}{c}{}
& \multicolumn{1}{c}{}
& \multicolumn{1}{c}{}
& \multicolumn{1}{c}{}
& \multicolumn{1}{c}{}
& \multicolumn{1}{c}{t-s}
& \multicolumn{1}{c}{}
& \multicolumn{1}{c}{}
& \multicolumn{1}{c}{}
& \multicolumn{1}{c}{}
& \multicolumn{1}{c}{}
& \multicolumn{1}{c}{}
\end{array}$$
\caption{$\Ext_{\A}^{s,t}(H^*(bo;Z/2), \Z/2)$}
\end{figure}

\subsection{The mod 2 Hurewicz map of $tmf$} \label{tmf section}

Let $tmf$ be the connective Topological Modular Forms spectrum, conceived of by Mike Hopkins \cite{hopkins}, and realized by Hopkins and an army of collaborators.  A good source of information, with many references, is \cite{TMF book}.  We compute the mod 2 Hurewicz image for $\Oinfty tmf$ using \thmref{conn thm}.

The $A$--module structure of the mod 2 cohomology is given by
$$H^*(tmf;\Z/2) = A/A\{Sq^1,Sq^2,Sq^4\} = A \otimes_{A(2)} \Z/2,$$ where
$A(2)$ is the sub(Hopf)--algebra generated by $Sq^1, Sq^2, Sq^4$.

In \cite{henriques}, Andr\'e Henriques works out the behavior of the Adams spectral sequence converging to $\pi_*(tmf)$, localized at 2.  We display what we will see is the relevant part of the $E_{\infty}$ page in Figure 3.

\begin{figure}[h]
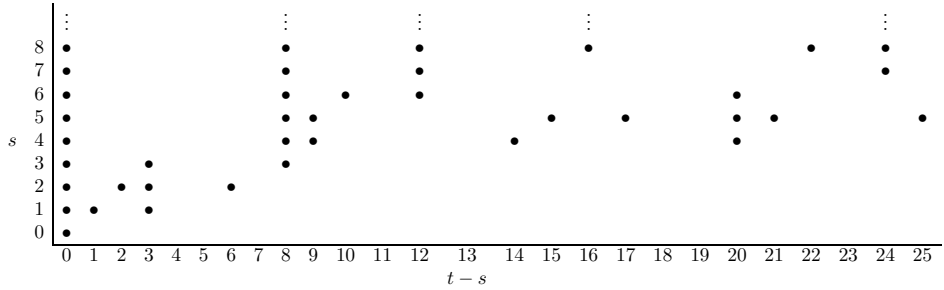

\scalebox{.67}
{$\begin{array}{cc|cccccccccccccc cccccccccccc}
 &  & \vdots  & &&&& & & &\vdots & &&&\vdots&   &&&\vdots&&&&&&&&\vdots&\\
 &8 & \bullet & &&&& & & &\bullet & &&&\bullet&  &&&\bullet &&&&&&\bullet&&\bullet&\\
 &7 & \bullet & && && &&  &\bullet  & &&&\bullet&  &&&&&&&&&&&\bullet&\\
 &6 & \bullet & &  &  && & & &\bullet &&\bullet&&\bullet &  &&&&&&&\bullet&&&&&\\
 &5 & \bullet & &  &  && & & &\bullet &\bullet &&&&&&\bullet&&\bullet&&&\bullet&\bullet&&&&\bullet\\
s&4 & \bullet & &  &  && & & &\bullet &\bullet &&&&&\bullet&&&&&&\bullet &&&&&\\
 &3 & \bullet & & &\bullet &  &  &  & &\bullet& &&&&  &&&&&&&&&&&&\\
 &2 & \bullet &&\bullet &\bullet & & &\bullet & & & &&&&  &&&&&&&&&&&&\\
 &1 & \bullet &\bullet &&\bullet  &  &  & & & & &&&&  &&&&&&&&&&&&\\
 &0 & \bullet &  &  &  &  &  & & & & &&&&  &&&&&&&&&&&&\\
\cline{3-28}
\multicolumn{1}{c}{}
& \multicolumn{1}{c}{}
& \multicolumn{1}{c}{0}
& \multicolumn{1}{c}{1}
& \multicolumn{1}{c}{2}
& \multicolumn{1}{c}{3}
& \multicolumn{1}{c}{4}
& \multicolumn{1}{c}{5}
& \multicolumn{1}{c}{6}
& \multicolumn{1}{c}{7}
& \multicolumn{1}{c}{8}
& \multicolumn{1}{c}{9}
& \multicolumn{1}{c}{10}
& \multicolumn{1}{c}{11}
& \multicolumn{1}{c}{12}
& \multicolumn{1}{c}{13}
& \multicolumn{1}{c}{14}
& \multicolumn{1}{c}{15}
& \multicolumn{1}{c}{16}
& \multicolumn{1}{c}{17}
& \multicolumn{1}{c}{18}
& \multicolumn{1}{c}{19}
& \multicolumn{1}{c}{20}
& \multicolumn{1}{c}{21}
& \multicolumn{1}{c}{22}
& \multicolumn{1}{c}{23}
& \multicolumn{1}{c}{24}
& \multicolumn{1}{c}{25}\\
\multicolumn{1}{c}{}
& \multicolumn{1}{c}{}
& \multicolumn{1}{c}{}
& \multicolumn{1}{c}{}
& \multicolumn{1}{c}{}
& \multicolumn{1}{c}{}
& \multicolumn{1}{c}{}
& \multicolumn{1}{c}{}
& \multicolumn{1}{c}{}
& \multicolumn{1}{c}{}
& \multicolumn{1}{c}{}
& \multicolumn{1}{c}{}
& \multicolumn{1}{c}{}
& \multicolumn{1}{c}{}
& \multicolumn{1}{c}{}
& \multicolumn{1}{c}{t-s}
& \multicolumn{1}{c}{}
& \multicolumn{1}{c}{}
& \multicolumn{1}{c}{}
& \multicolumn{1}{c}{}
& \multicolumn{1}{c}{}
& \multicolumn{1}{c}{}
\end{array}$}
\caption{The $E_{\infty}$ term of the mod 2 Adams spectral sequence for $tmf$}
\end{figure}

As before, one obtains the chart for the $E_{\infty}$ page for $tmf\langle 1 \rangle$ from the chart in Figure 3 by removing the $0$th column and dropping $s$ values by one.  Using that $A(2)$ has top class in degree 23, \lemref{vanishing lemma} shows that $\Ext_A^{s,t}(H^*(tmf\langle 1 \rangle;\Z/2), \Z/2) = 0$ unless $t-s \leq 22(s+1)$. If also $t-s \geq 2^s$, it follows that $s \leq 7$ and $t-s \leq 176$, inside the range of the charts in \cite{henriques}.  Figure 4 shows what we need of the $E_{\infty}$ page for $tmf\langle 1 \rangle$: the  classes satisfying the condition $t-s \geq 2^s$ are depicted as $\textcolor{cyan}{\bullet}$ or $\textcolor{magenta}{\bullet}$.

\begin{thm} \label{tmf thm 2} In Figure 4, classes in $\pi_*(tmf\langle 1\rangle)$ represented by a label $\textcolor{cyan}{\bullet}$ have nonzero Hurewicz image, while classes represented by a
label $\textcolor{magenta}{\bullet}$ map to zero under the Hurewicz map.
\end{thm}

This, of course, proves \thmref{tmf thm}.

\begin{figure}[h]
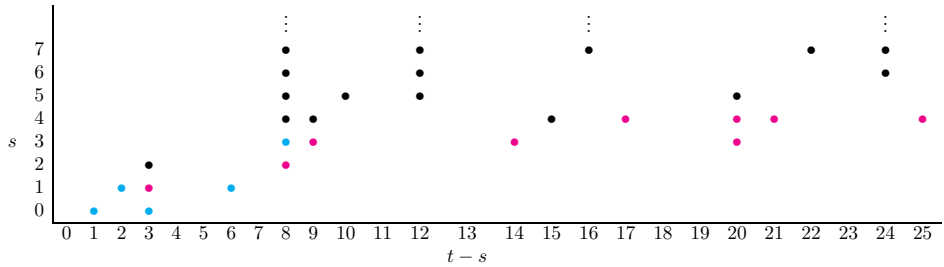

\scalebox{.67}
{$\begin{array}{cc|cccccccccccccc cccccccccccc}
 &  &  & &&&& & & &\vdots & &&&\vdots&   &&&\vdots&&&&&&&&\vdots&\\
 &7 &  & &&&& & & &\bullet & &&&\bullet&  &&&\bullet &&&&&&\bullet&&\bullet&\\
 &6 &  & && && &&  &\bullet  & &&&\bullet&  &&&&&&&&&&&\bullet&\\
 &5 &  & &  &  && & & &\bullet &&\bullet&&\bullet &  &&&&&&&\bullet&&&&&\\
 &4 &  & &  &  && & & &\bullet &\bullet &&&&&&\bullet&&\textcolor{magenta}{\bullet}&&&\textcolor{magenta}{\bullet}
 &\textcolor{magenta}{\bullet}&&&&\textcolor{magenta}{\bullet}\\
s&3 &  & &  &  && & & &\textcolor{cyan}{\bullet} &\textcolor{magenta}{\bullet} &&&&&\textcolor{magenta}{\bullet}&&&&&&\textcolor{magenta}{\bullet}&&&&&\\
 &2 &  & & &\bullet &  &  &  & &\textcolor{magenta}{\bullet}& &&&&  &&&&&&&&&&&&\\
 &1 &  &&\textcolor{cyan}{\bullet} &\textcolor{magenta}{\bullet} & & &\textcolor{cyan}{\bullet} & & & &&&&  &&&&&&&&&&&&\\
 &0 &  &\textcolor{cyan}{\bullet} &&\textcolor{cyan}{\bullet} &  &  & & & & &&&&  &&&&&&&&&&&&\\
\cline{3-28}
\multicolumn{1}{c}{}
& \multicolumn{1}{c}{}
& \multicolumn{1}{c}{0}
& \multicolumn{1}{c}{1}
& \multicolumn{1}{c}{2}
& \multicolumn{1}{c}{3}
& \multicolumn{1}{c}{4}
& \multicolumn{1}{c}{5}
& \multicolumn{1}{c}{6}
& \multicolumn{1}{c}{7}
& \multicolumn{1}{c}{8}
& \multicolumn{1}{c}{9}
& \multicolumn{1}{c}{10}
& \multicolumn{1}{c}{11}
& \multicolumn{1}{c}{12}
& \multicolumn{1}{c}{13}
& \multicolumn{1}{c}{14}
& \multicolumn{1}{c}{15}
& \multicolumn{1}{c}{16}
& \multicolumn{1}{c}{17}
& \multicolumn{1}{c}{18}
& \multicolumn{1}{c}{19}
& \multicolumn{1}{c}{20}
& \multicolumn{1}{c}{21}
& \multicolumn{1}{c}{22}
& \multicolumn{1}{c}{23}
& \multicolumn{1}{c}{24}
& \multicolumn{1}{c}{25}\\
\multicolumn{1}{c}{}
& \multicolumn{1}{c}{}
& \multicolumn{1}{c}{}
& \multicolumn{1}{c}{}
& \multicolumn{1}{c}{}
& \multicolumn{1}{c}{}
& \multicolumn{1}{c}{}
& \multicolumn{1}{c}{}
& \multicolumn{1}{c}{}
& \multicolumn{1}{c}{}
& \multicolumn{1}{c}{}
& \multicolumn{1}{c}{}
& \multicolumn{1}{c}{}
& \multicolumn{1}{c}{}
& \multicolumn{1}{c}{}
& \multicolumn{1}{c}{t-s}
& \multicolumn{1}{c}{}
& \multicolumn{1}{c}{}
& \multicolumn{1}{c}{}
& \multicolumn{1}{c}{}
& \multicolumn{1}{c}{}
& \multicolumn{1}{c}{}
\end{array}$}
\caption{The $E_{\infty}$ term of the mod 2 Adams spectral sequence for $tmf\langle 1\rangle$}
\end{figure}

We work our way through the proof.   Through dimension 6, the map $S \ra tmf$ is an equivalence.  Thus the associated Hurewicz homomorphisms agree in this range.  The classes labeled $\textcolor{cyan}{\bullet}$ in dimensions 1,2,3, and 6 correspond to $\eta, \eta^2, \nu$, and $\nu^2$ in  $\pi_*(S)$.  These are well known to have nonzero Hurewicz image in $H_*(QS^0;\Z/2)$.

The classes labeled $\textcolor{magenta}{\bullet}$ in $(t-s,s)$ bidegrees (8,2), (14,3), and (20,3) can be represented by the image under $\pi_*(S) \ra \pi_*(tmf)$ \cite[\S 8]{bauer} of elements $\epsilon$, $\kappa$, and $\bar{\kappa}$ (named as in Appendix A3 of \cite{ravenel}) with zero Hurewicz image in $H_*(QS^0;\Z/2)$, and thus also in $H_*(\Oinfty tmf;\Z/2)$.

If $c_4 \in \pi_8(tmf\langle 1\rangle) \simeq \Z \oplus \Z/2$ is represented in bidegree (8,3) (and thus bidegree (8,4) in $\pi_*(tmf)$), it will have infinite order and map to a generator of $\pi_8(ko)$ under the canonical map $tmf \ra ko$. (Compare with \cite[Thm.1.2]{lawson naumann}.)  Since such a generator has nonzero Hurewicz image in $H_8(BO;\Z/2)$, $h(c_4) \in H_8(\Oinfty tmf;\Z/2)$ will also be nonzero.

We now make a simple, but useful, observation about general Hurewicz maps $\pi_*(X) \ra R_*(\Oinfty X)$.
\begin{lem} \label{simple lem}  If $\gamma: S^a \ra X$ factors as $S^a \xra{\alpha} Y \xra{\beta} X$, and $h(\alpha) = 0$ in $R_a(\Oinfty Y)$, then $h(\gamma) = 0$ in $R_a(\Oinfty X)$.
\end{lem}
\begin{proof} The adjoint $\gamma_0: S^a \ra \Oinfty X$ of $\gamma$ factors as $S^a \xra{\alpha_0} \Oinfty Y \xra{\Oinfty \beta} \Oinfty X$, where $\alpha_0$ is adjoint to $\alpha$.  So $\alpha_{0*} = 0$ implies $\gamma_{0*} = 0$.
\end{proof}
In the situation of the theorem, the lemma applies to the classes labeled $\textcolor{magenta}{\bullet}$ in Figure 4 in  bidegrees (3,1), (9,3), (17,4), (20,4), and (21,4), which respectively correspond to $2\nu$ (viewed as $\nu \circ 2$), $\epsilon \eta$, $\kappa \nu$, $2\bar{\kappa}$,  and $\bar{\kappa}\nu$.

It remains to determine show that if $\gamma \in \pi_{25}(tmf)$ is represented in Figure 4 by the label $\textcolor{magenta}{\bullet}$ in bidegree (25,4), then $h(\gamma)=0$.  Tilman Bauer \cite[Prop.8.4(1)]{bauer} tells us that one such element satisfies $\gamma \in \langle \bar{\kappa}, \nu, \eta\rangle$. It thus factors as $S^{25} \ra (S^{20} \cup_{\nu} D^{24}) \ra tmf$, and the lemma applies again.

\subsection{The mod $p$ Hurewicz map for $BP\langle n \rangle_c$ and $BP_c$} \label{BP section}

The $p$--local Brown--Peterson spectrum $BP$ is a commutative ring spectrum with mod $p$ cohomology
$$H^*(BP;\Z/p) = A \otimes_{E} \Z/p,$$
where $E$ is an exterior algebra on the Milnor classes $Q_0, Q_1, Q_2, \dots \in A$ with $|Q_i|=2p^i-1$ \cite{brown peterson}. It follows that
$$ \Ext_{A}^{*,*}(H^*(BP;\Z/p),\Z/p) = \Ext_{E}^{*,*}(\Z/p,\Z/p) = \Z/p[v_0, v_1, v_2, \dots],$$
with $v_i \in \Ext_{A}^{1,2p^i-1}$.  As all elements are in even total degree, the Adams spectral sequence collapses, and one sees that
$$ \pi_*(BP) = \Z_{(p)}[v_1, v_2, \dots].$$

We begin this subsection by showing that this information implies that the hypothesis of finite presentation can't be relaxed to finite generation in our Finiteness Theorem, \thmref{finiteness thm}.

\begin{prop} \label{BP prop}  Under the unstable Hurewicz map $$h_*: \pi_*(BP) \ra H_*(\Oinfty BP;\Z/p),$$ $h_*(v_i) \neq 0$ for all $i$.
\end{prop}
\begin{proof}  Let $Y$ be the 0--connected cover of $BP$, so that $\Oinfty Y = \Oinfty_0 BP$, and thus $BP$ and $Y$ have the same unstable Hurewicz image in positive degrees.  As $Y$ is the fiber of $BP \ra H\Z_{(p)}$, it follows easily that one gets the $E_{\infty}$ page of the Adams spectral sequence for $Y$ from that of $BP$ by removing the $0$th column and dropping $s$ values by one.  In particular, all of the $v_i$ for $i \geq 1$ now appear in $E_{\infty}^{0,*}$, and thus are in the image of the stable Hurewicz map $h^{st}_*: \pi_*(Y) \ra H_*(Y;\Z/p)$.  But then these same classes have nonzero Hurewicz image in $H_*(\Oinfty Y;\Z/p)$ as well.
\end{proof}

For $n \geq 0$, let $BP\langle n \rangle$ denote the $n$th Johnson-Wilson spectrum \cite{johnson wilson}, a commutative ring spectrum with mod $p$ cohomology
$$H^*(BP\langle n \rangle;\Z/p) = A \otimes_{E(n)} \Z/2,$$
where $E(n)$ is the exterior algebra on the classes $Q_0, \dots, Q_n \in A$.  As before, it
follows that
$$ \Ext_{A}^{*,*}(H^*(BP\langle n \rangle;\Z/p),\Z/p) = \Ext_{E(n)}^{*,*}(\Z/p,\Z/p) = \Z/p[v_0, \dots, v_n],$$
and
$$ \pi_*(BP\langle n \rangle) = \Z_{(p)}[v_1, \dots, v_n].$$

We apply the Connectivity Theorem to $X= \Sigma^c BP\langle n \rangle$, so that $\Oinfty X = BP\langle n\rangle_c$, to prove \thmref{BPn thm}, which we restate here.

\begin{thm} If $c> 2(p^n-1)/(p-1)$, then the mod p Hurewicz image of $BP\langle n \rangle_c$ is just the bottom $\Z/p$ in degree $c$.  Thus $BP\langle n \rangle_c$ is atomic: it can't be written as a product of two spaces in a nontrivial way.
\end{thm}
\begin{proof}  If $X= \Sigma^c BP\langle n \rangle$, then, for fixed $s$,
$$ \max\{t-s \ | \ \Ext_A^{s,t}(H^*(X;\Z/p),\Z/p) \neq 0\} = c+|v_n|^s =c+ 2s(p^n-1).$$
Thus the Connectivity Theorem implies that the Hurewicz image for $X$ will just be the bottom degree vector space $\widetilde H_c(\Oinfty X;\Z/p) \simeq \Z/p$ if  for all $s \geq 1$, $c+2s(p^n-1) < cp^s$.  This rearranges to say that $c$ is such that
$$c> 2(p^n-1)\frac{s}{(p^s-1)}$$
for all $s \geq 1$, and this is clearly just the condition $c> 2(p^n-1)/(p-1)$.
\end{proof}

Note that we have proved this theorem using only that $BP\langle n \rangle$ has mod $p$ cohomology isomorphic to $A//E(n)$.

In \cite{wilson I}, Steve Wilson does a careful study of the cohomology rings $H^*(BP\langle n\rangle_c; \Z/p)$ and then uses this in \cite{wilson II} to give complete product decompositions of the space $BP_c$ and $BP\langle n \rangle_c$ (including a rather different proof of the atomicity result above).   As we now explain, these decompositions allow us to completely determine the mod $p$ Hurewicz map for the spaces $BP\langle n \rangle_c$ and $BP_c$, for all $c\geq 0$.  

In the next two theorems, let $BP\langle \infty \rangle$ denote $BP$, and, for $c>0$, let $m(c)$ be defined so that
$$ 2(p^{m(c)}-1)/(p-1)<c \leq 2(p^{m(c)+1}-1)/(p-1).$$

\begin{thm} \label{BP thm} For all $0\leq n \leq \infty$, the following hold. \\

\noindent{\bf (a)} \ If $c>0$, $\ker\{h_*: \pi_*(BP\langle n \rangle) \ra H_{*+c}(BP\langle n \rangle_c;\Z/p)\}$ is the subgroup $$F_2\pi_*(BP\langle n \rangle) + \langle v_m \ | \ m\leq m(c)\rangle.$$   Thus $\im h$ has a basis given by $h_*(1)$ and the elements $h_*(v_m)$ with $m(c)<m\leq n$. \\

\noindent{\bf (b)} \ If $p=2$ then, in positive dimensions,
$$\ker\{\pi_*(BP\langle n \rangle) \xra{h_*} H_{*}(BP\langle n \rangle_0;\Z/p)\}$$
is the subgroup $F_3\pi_*(BP\langle n \rangle) + \langle v_iv_j \ | \ i<j\rangle$.  Thus, in positive dimensions, $\im h_*$ has a basis given by the elements $h_*(v_m)$ and $h_*(v_m^2)$with $m\leq n$.\\

\noindent{\bf (c)} \ If $p$ is odd then, in positive dimensions,
$$\ker\{h_*: \pi_*(BP\langle n \rangle) \ra H_{*}(BP\langle n \rangle_0;\Z/p)\}$$
equals $\ker\{h_*: \pi_*(BP\langle n \rangle) \ra H_{*+1}(BP\langle n \rangle_1;\Z/p)\}$. Thus, in positive dimensions, $\im h_*$ has a basis given by the elements $h_*(v_m)$ with $m\leq n$.
\end{thm}

We start the proof of this with another general lemma about the Hurewicz map $h_*: \pi_*(X) \ra R_*(X)$, with $R$ as in \thmref{lifting thm}.  Lemmas like this one have also been exploited by Hadi Zare \cite{zare}.

\begin{lem} \label{simple ring lem}  Let $X$ be a connective ring spectrum and let $\alpha \in \pi_a(X)$ and $\beta \in \pi_b(X)$ be elements with $a \geq b$ and $\beta$ of positive $R$--based Adams filtration.  Then, localized away from $(p-1)!$, $h_*(\alpha \beta)=0$ in $R_{a+b+c}(X_c)$ unless $p=2$, $c=0$, and $a=b$.
\end{lem}
\begin{proof}  Using the ring structure on $X$, $\alpha \beta$ can be written as $S^{a+b} \xra{\Sigma^a \beta} \Sigma^a X \xra{\bar \alpha} X$, where $\bar \alpha$ is $\Sigma^a X \xra{\alpha \sm 1_X} X\sm X \xra{m} X$.  By \lemref{simple lem}, $h_*(\alpha \beta)=0$ in $R_{a+b+c}(X_c)$ if $h_*(\beta)=0$ in $R_{a+b+c}(X_{a+c})$.  The $s=1$ case of our Connectivity Theorem (or, much more simply, the Freudenthal Theorem, except when $p>2$ and $c=0$) now applies to show that $h_*(\beta)=0$ if $a+b+c < p(a+c)$, which rewrites as $b<(p-1)(a+c)$.
\end{proof}

Applying this to the mod $p$ homology Hurewicz map $h_*$ when $X$ is $BP\langle n \rangle$ immediately shows that almost all monomials in the $v_i$ have zero Hurewicz image in $H_*(BP\langle n \rangle_c;\Z/p)$. We conclude that, unless $p=2$ and $c=0$, the elements $h_*(v_m)$ span the Hurewicz image in $H_*(BP\langle n \rangle_c;\Z/p)$, with the elements $h_*(v_m^2)$ also possibly being nonzero in this last case.  As these elements are all in distinct degrees, we will be able to precisely determine the Hurewicz image if $BP\langle n \rangle_c$ can be written as a product of spaces of the sort appearing in \thmref{BPn thm}.

Assembling decomposition results from \cite{wilson II} shows that this is the case.  Recall that $|v_m|=2(p^m-1)$.

\begin{thm}[\cite{wilson II}]  For $0 \leq n \leq \infty$, the space $BP\langle n \rangle_c$ has a decomposition as a product of atomic spaces as follows. \\

\noindent{\bf (a)} If $c>0$ and $n\leq m(c)$, then $BP\langle n \rangle_c$ is atomic. \\

\noindent{\bf (b)} If $c>0$ and $n > m(c)$, then
$$BP\langle n \rangle_c \simeq BP\langle m(c) \rangle_c \times \prod_{m=m(c)+1}^n BP\langle m \rangle_{c+|v_m|}.$$

\noindent{\bf (c)} If $p$ is odd, then
$$BP\langle n \rangle_0 \simeq \Z_{(p)} \times \prod_{m=1}^n BP\langle m \rangle_{|v_m|}.$$

\noindent{\bf (d)} If $p=2$, then
$$BP\langle n \rangle_0 \simeq \Z_{(2)} \times \prod_{m=1}^n (BP\langle m-1 \rangle_{|v_m|} \times BP\langle m \rangle_{2|v_m|}).$$
\end{thm}

Part (a) is just our \thmref{BPn thm} again and appears as Corollary 5.3 of \cite{wilson II}.  Part (b) follows by combining Theorem 5.4 and Corollary 5.5 of \cite{wilson II}.  Looping the decomposition of part (b) when $c=1$ implies parts (c) and (d), noting that, when $p=2$, \cite[Cor.5.5]{wilson II} tells us that
$$BP\langle m \rangle_{|v_m|}  \simeq BP\langle m-1 \rangle_{|v_m|} \times BP\langle m \rangle_{2|v_m|}.$$

\end{document}